\documentclass[]{imsart}
\usepackage{amsthm,amsmath}
\usepackage[colorlinks,citecolor=blue,urlcolor=blue]{hyperref}

\usepackage{amsmath,amsfonts,amssymb,amsthm}

\usepackage{enumitem}

\RequirePackage{amsthm,amsmath,amsfonts,amssymb}
\RequirePackage[numbers]{natbib}
\RequirePackage[colorlinks,citecolor=blue,urlcolor=blue]{hyperref}
\RequirePackage{graphicx}
\usepackage{lineno,hyperref}

\startlocaldefs
\modulolinenumbers[5]

\usepackage[utf8]{inputenc}
\usepackage[english]{babel}

\usepackage[margin=1in]{geometry}
\usepackage{bm}
\usepackage{graphicx}
\usepackage{amssymb,amsmath,amsthm,mathrsfs}
\usepackage{epstopdf}
\usepackage{algorithm}
\usepackage{amsfonts,delarray}
\usepackage{algorithm,algcompatible,amsmath}
\usepackage{rotating,color}
\usepackage{subfigure}
\usepackage{multirow}
\usepackage{titlesec,textcase}
\usepackage{longtable}
\usepackage{bbm}
\usepackage{rotating}

\newtheorem{Theorem}{Theorem}[section]
\newtheorem{Lemma}[Theorem]{Lemma}

\theoremstyle{definition}

\RequirePackage[normalem]{ulem} 

\definecolor{rp}{RGB}{83,54,106}

\def\boxit#1{\vbox{\hrule\hbox{\vrule\kern6pt\vbox{\kern6pt#1\kern6pt}\kern6pt\vrule}\hrule}}

\endlocaldefs

\begin{document}
\begin{frontmatter}
\title{Sharp detection boundaries on testing dense subhypergraph}
\runtitle{Sharp detection boundaries on testing dense subhypergraph}
\runauthor{M. Yuan and Z. Shang}
\begin{aug}
\author[A]{\fnms{Mingao} \snm{Yuan}\ead[label=e1]{mingao.yuan@ndsu.edu}}
\and
\author[B]{\fnms{Zuofeng} \snm{Shang}\ead[label=e2]{zshang@njit.edu}}
\address[A]{Department of Statistics,
North Dakota State University,
\printead{e1}}

\address[B]{Department of Mathematical Sciences,
New Jersey Institute of Technology,
\printead{e2}}
\end{aug}

\begin{abstract}
We study the problem of testing the existence of a dense subhypergraph.
The null hypothesis is an Erd\"{o}s-R\'{e}nyi uniform random hypergraph
and the alternative hypothesis is a uniform random hypergraph
that contains a dense subhypergraph. 
We establish sharp detection boundaries in both scenarios: (1) the edge probabilities are known; (2) the edge probabilities are unknown. 
In both scenarios, sharp detectable boundaries are characterized by 
the appropriate model parameters. 
Asymptotically powerful tests are provided when the model parameters 
fall in the detectable regions. Our results indicate that the detectable regions for general hypergraph models are dramatically different from their
graph counterparts.
\end{abstract}



\end{frontmatter}

\section{Introduction}
\label{S:1}
Suppose we observe an \textit{undirected} $m$-uniform hypergraph  $(\mathcal{V},\mathcal{E})$ on $N:=|\mathcal{V}|$ vertices with an edge set $\mathcal{E}$. Each edge in $\mathcal{E}$ consists of exactly $m$ vertices. 
In particular, $m=2$ degenerates to ordinary graphs. Without loss of generality, denote $\mathcal{V}=[N]:=\{1,2,\ldots,N\}$.
The corresponding adjacency tensor is an $m$-dimensional 0-1 symmetric array $A\in(\{0,1\}^N)^{\otimes m}$ satisfying $A_{i_1i_2\ldots i_m}=1$ if and only if $\{i_1,i_2,\ldots, i_m\}\in \mathcal{E}$.
By symmetry we mean that $A_{i_1i_2\ldots i_m}=A_{j_1j_2\ldots j_m}$ whenever $i_1,i_2,\ldots,i_m$ is a permutation of $j_1,j_2,\ldots,j_m$. 
Given $A$, we are interested in testing the existence of 
a dense subhypergraph, which can be formulated as a hypothesis testing problem:
\begin{equation}\label{hypothesis}
H_0: A\sim \mathcal{H}_m(N,p_0)\,\,\,\,\text{vs.}\,\,\,\, H_1: A\sim \mathcal{H}_m(N,p_0; n,p_1),
\end{equation}
where $0<p_0<p_1\le 1$ are known.
The null hypothesis in (\ref{hypothesis}) says that $A$ follows an
Erd\"{o}s-R\'{e}nyi $m$-uniform random hypergraph with edge rate $p_0$. Equivalently, $A_{i_1i_2\ldots i_m}$ for $1\leq i_1<i_2<\cdots <i_m\leq N$ are independent and identically distributed 
Bernoulli variables with $\mathbb{P}(A_{i_1i_2\ldots i_m}=1)=p_0$. The alternative hypothesis in (\ref{hypothesis})
says that $A_{i_1i_2\ldots i_m}$ for $1\leq i_1<i_2<\cdots <i_m\leq N$ are independent Bernoulli variables and
there exists a subset $S\subset\mathcal{V}$ with $|S|=n$ such that $\mathbb{P}(A_{i_1i_2\ldots i_m}=1)=p_1$ if the distinct vertices $i_1,i_2,\ldots,i_m$ all belong to $S$, and $\mathbb{P}(A_{i_1i_2\ldots i_m}=1)=p_0$ otherwise. 
The assumption $p_1>p_0$ implies that the vertices within $S$ are more possibly connected. When $p_1=1$, $S$ is called the hypergraphic planted clique (HPC),
and problem (\ref{hypothesis}) becomes testing the existence of HPC.
We borrow the terminology HPC from \cite{LZ20}.
When $m=2$, Arias-Castro and Verzelen \cite{AV14} established sharp detection boundaries for testing (\ref{hypothesis}). For general $m$, Bollob\'{a}s and Erd\"{o}s \cite{BE76} considered testing HPC and provided a sufficient condition in terms of $n,N,m,p_0$ under which the exhaustive search algorithm is successful.
However, it remains unknown whether sharp detection boundaries for testing
(\ref{hypothesis}) still exist for arbitrary $m$.

In this paper, we positively answer the above question and provide sharp
detection boundaries for testing the problem (\ref{hypothesis}) for arbitrary $m\ge2$.
A statistical test $T$ for testing (\ref{hypothesis}) is a 0-1 valued function of the observed adjacency tensor $A$ satisfying $T=1$ if and only if $H_0$ is rejected. The risk of $T$ is defined as
\[
\gamma_N(T)=\mathbb{P}_0(T=1)+\max_{|S|=n}\mathbb{P}_S(T=0),
\]
where $\mathbb{P}_0$ and $\mathbb{P}_S$ denote the probability measures
under $H_0$ and $H_1$, respectively.
A test $T$ is said to be asymptotically powerful (or asymptotically powerless) if 
$\gamma_N(T)\to0$ (or $\gamma_N(T)\to1$).
If an asymptotically powerful test exists, we say that the dense subhypergraph is detectable; otherwise, it is undetectable.
As an initial stage, we show that the sufficient condition provided in
\cite{BE76} is also necessary (see Theorem \ref{clique0}). 
As a byproduct, we propose the hypergraphic clique number test (HCNT) 
for testing HPC which is proven asymptotically powerful. 
We next consider the problem (\ref{hypothesis}) for general $0<p_0<p_1<1$ and derive sharp detection boundary in terms of $n,N,m,p_0,p_1$
(see Theorems \ref{thm:1} and \ref{thm:2}).
We propose either the hypergrahic total degree test (HTDT) or
hypergraphic scan test (HST) both being proven asymptotically powerful. See Table \ref{p0known} 
for a summary of our results for testing (\ref{hypothesis}) in the two regimes $p_1=1$ and $p_1\in(0,1)$, including the corresponding asymptotically powerful tests if they exist, in which
$H_{p}(q)$ is the Kullback-Leibler divergence from $\text{Bern}(q)$ to $\text{Bern}(p)$ defined as
$H_{p}(q)=q\log\frac{q}{p}+(1-q)\log\frac{1-q}{1-p}$, for $p, q\in (0,1)$.
\begin{table}[H]
\vspace{-2mm}
	\centering
	\begin{tabular}{|p{2cm} p{6cm} p{2cm} p{4cm}|}\hline
	 &  & Detectibility & Asymptotically powerful test \\
	\hline
	\multirow{ 2}{*}{$0<p_0<p_1=1$}&$n< \left(m!\log_{\frac{1}{p_0}}N\right)^{\frac{1}{m-1}}$&Undetectable&None\\
	&$n>\left(m!\log_{\frac{1}{p_0}}N\right)^{\frac{1}{m-1}}$&Detectable&HCNT\\ \hline
	\multirow{ 3}{*}{$0<p_0<p_1<1$}&$\frac{p_1-p_0}{\sqrt{p_0}}\ll \left(\frac{N}{n^2}\right)^{\frac{m}{2}}$ and $H_{p_0}(p_1)<\frac{m!\log\frac{N}{n}}{n^{m-1}}$ & Undetectable        &  None \\
 &$\frac{p_1-p_0}{\sqrt{p_0}}\gg \left(\frac{N}{n^2}\right)^{\frac{m}{2}}$    &	 Detectable         & HTDT\\
	 &$H_{p_0}(p_1)>\frac{m!\log\frac{N}{n}}{n^{m-1}}$    & Detectable         & HST\\
	\hline
\end{tabular}
	\caption{\it Detection boundaries for testing (\ref{hypothesis}) when $p_0, p_1$ are known and the corresponding asymptotically powerful tests.}	
\label{p0known}
\vspace{-5mm}
\end{table}

As a second stage, assume that $p_0$ and $p_1$ are both unknown, and so
the hypotheses in (\ref{hypothesis}) become composite.
Specifically, suppose there exist $\mathcal{C}_0\subset(0,1)$ and $\mathcal{C}_1\subset(0,1)^2$ such that $p_0\in\mathcal{C}_0$ in $H_0$ and
$(p_0,p_1)\in\mathcal{C}_1$ in $H_1$.
In this case, the risk of a test $T$ becomes
\begin{equation}\label{new:risk}
\gamma_N^*(T)=\sup_{p_0\in\mathcal{C}_0}\mathbb{P}_0(T=1)+
\sup_{(p_0,p_1)\in\mathcal{C}_1}\max_{|S|=n}\mathbb{P}_S(T=0).
\end{equation}
A test $T$ is said to be asymptotically powerful (or asymptotically powerless) if 
$\gamma_N^*(T)\to0$ (or $\gamma_N^*(T)\to1$).
Let 
\begin{equation}\label{p0:prime}
p_0^\prime=\frac{N^{(m)}p_0-n^{(m)}p_1}{N^{(m)}-n^{(m)}},
\end{equation}
where $N^{(m)}:={N\choose m}$.
We derive sharp detection boundary for testing (\ref{hypothesis}) which are characterized
in terms of $n,N,m,p_0',p_1$
(see Theorems \ref{thm:4} and \ref{thm:6})
We also propose hypergraphic scan test (HST), hypergraphic loose 2-path test (HL2-PT) or hypergraphic tight 2-path test (HT2-PT) for detecting the dense subhypergraph which are proven asymptotically powerful. See Table \ref{p0unknown} for a summary of our results including the corresponding asymptotically powerful tests.

\begin{table}[H]
\vspace{-2mm}
	\centering
	\begin{tabular}{|p{6.5cm} p{2cm} p{4cm}|}
	\hline
	  & Detectibility & Asymptotically powerful test \\
	\hline
	$\frac{p_1-p_0^\prime}{\sqrt{p_0^\prime}}\ll \left(\frac{N}{n^2}\right)^{\frac{m+1}{4}}$ and $H_{p_0^\prime}(p_1)<\frac{m!\log\frac{N}{n}}{n^{m-1}}$ & Undetectable        &  None\\
 $\frac{p_1-p_0^\prime}{\sqrt{p_0^\prime}}\gg \left(\frac{N}{n^2}\right)^{\frac{m+1}{4}}$ and $n^2\succ N$   &	 Detectable         &  HL2-PT\\
 $\frac{p_1-p_0^\prime}{\sqrt{p_0^\prime}}\gg \left(\frac{N}{n^2}\right)^{\frac{m+1}{4}}$ and $n^2\ll N$   &	 Detectable         & HT2-PT\\
	 $H_{p_0^\prime}(p_1)>\frac{m!\log\frac{N}{n}}{n^{m-1}}$    & Detectable         & HST\\
	\hline
\end{tabular}
	\caption{\it Detection boundary for testing (\ref{hypothesis}) when $p_0, p_1$ are unknown and the corresponding asymptotically powerful tests.}	
\label{p0unknown}
\vspace{-5mm}
\end{table}
An interesting byproduct is the dramatic difference between testing 
dense subgraph and testing dense subhypergraph as displayed in Figures
\ref{HPCphase} to \ref{unknownHperphase}. For instance, testing HPC
seems easier than testing PC since the former has smaller undetectable regions than the latter; see Figure \ref{HPCphase}.
When $n\succ\sqrt{N}$, testing general dense subhypergraph seems easier than 
testing dense subgraph; see Figures \ref{subHperphase} and \ref{unknownsubHperphase}. However, the two testing problems are no longer comparable when $n=o(\sqrt{N})$; see Figures \ref{Hperphase} and \ref{unknownHperphase}.

This work is among the recent surge in theoretical study of hypergraph models 
primarily from statistical perspectives. 
Other developments in this field are summarized below.
Assuming the existence of a dense subhypergraph,
\cite{LJY15, K11,ZHS06, BGK17,HWC17, CDKKR18,LZ20} proposed 
various detection algorithms. 
In stochastic block models,
\cite{GD14,GD17,ACKZ15,B93,CK10,KBG17,LCW17,R09,RP13,KSX20,FP16,ALS16,ALS18} proposed various algorithms for detecting the underlying communities, and
\cite{YLFS20} established sharp phase transition phenomenon for 
testing the existence of the communities. 
The readers are referred
to the survey paper \cite{BTYZQ21} for more
references. 

\section{Detection of HPC}\label{clique}

Consider the problem of detecting HPC, i.e., $p_1=1$ in (\ref{hypothesis}). 
Define the hypergraphic clique number test (HCNT) as 
$\omega_N=\max\{0\leq k\leq N: V\subset[N],W_V=k^{(m)},|V|=k\}$, where
$W_V:=\sum_{\substack{i_1,\ldots,i_m\in V\\ i_1<\cdots<i_m}} A_{i_1i_2\ldots i_m}$ is the number of edges in the subhypergraph restricted to $V$.
The following theorem provides a sharp detection boundary characterized by $n,N,m,p_0$.

\begin{Theorem}\label{clique0}
Let $m\ge2$. The following results hold.
\begin{enumerate}[label={(\Roman*)},itemindent=1em]
\item 
All tests are asymptotically powerless if for a constant $\epsilon\in(0,1)$,
\[
n\leq \left[m!(1-\epsilon)\log_{\frac{1}{p_0}}N\right]^{\frac{1}{m-1}}.
\]\label{thm:1:I} 
\item  The HCNT is asymptotically powerful if for a constant $\epsilon\in(0,1)$,
\begin{equation}\label{BE:cond}
n\geq \left[m!(1+\epsilon)\log_{\frac{1}{p_0}}N\right]^{\frac{1}{m-1}}.
\end{equation}
\label{thm:1:II}
\end{enumerate}
\end{Theorem}

For Part \ref{thm:1:I}, one can show that $H_0$ and $H_1$ are asymptotically mutually contiguous so that all tests are asymptotically powerless.
For Part \ref{thm:1:II}, one can show that $\gamma_N(\omega_N)\to1$ so that HCNT is asymptotically powerful. 
Bollob\'{a}s and Erd\"{o}s \cite{BE76} showed that HPC can be detected through exhaustive search under (\ref{BE:cond}), so Part \ref{thm:1:II} is nothing but a restatement of their result in hypothesis testing framework.
Part \ref{thm:1:I} says that this condition is necessary to ensure the existence of a successful test. 

An interesting byproduct of Theorem \ref{clique0} is to provide an understanding on the relationship between testing PC ($m=2$) and HPC ($m>2$).
Figure \ref{HPCphase} demonstrates detection regions when $m=2,3$ and $p_0=1/2$. The green color indicates regions where PC or HPC are detectable,
and the red color indicates regions where they are undetectable.
Interestingly, $m=3$ yields a larger detectable region than $m=2$, displaying a dramatic difference between testing PC and testing HPC. 
\begin{figure}[htp] 
\vspace{-3mm}
\centering
\includegraphics[scale=0.39]{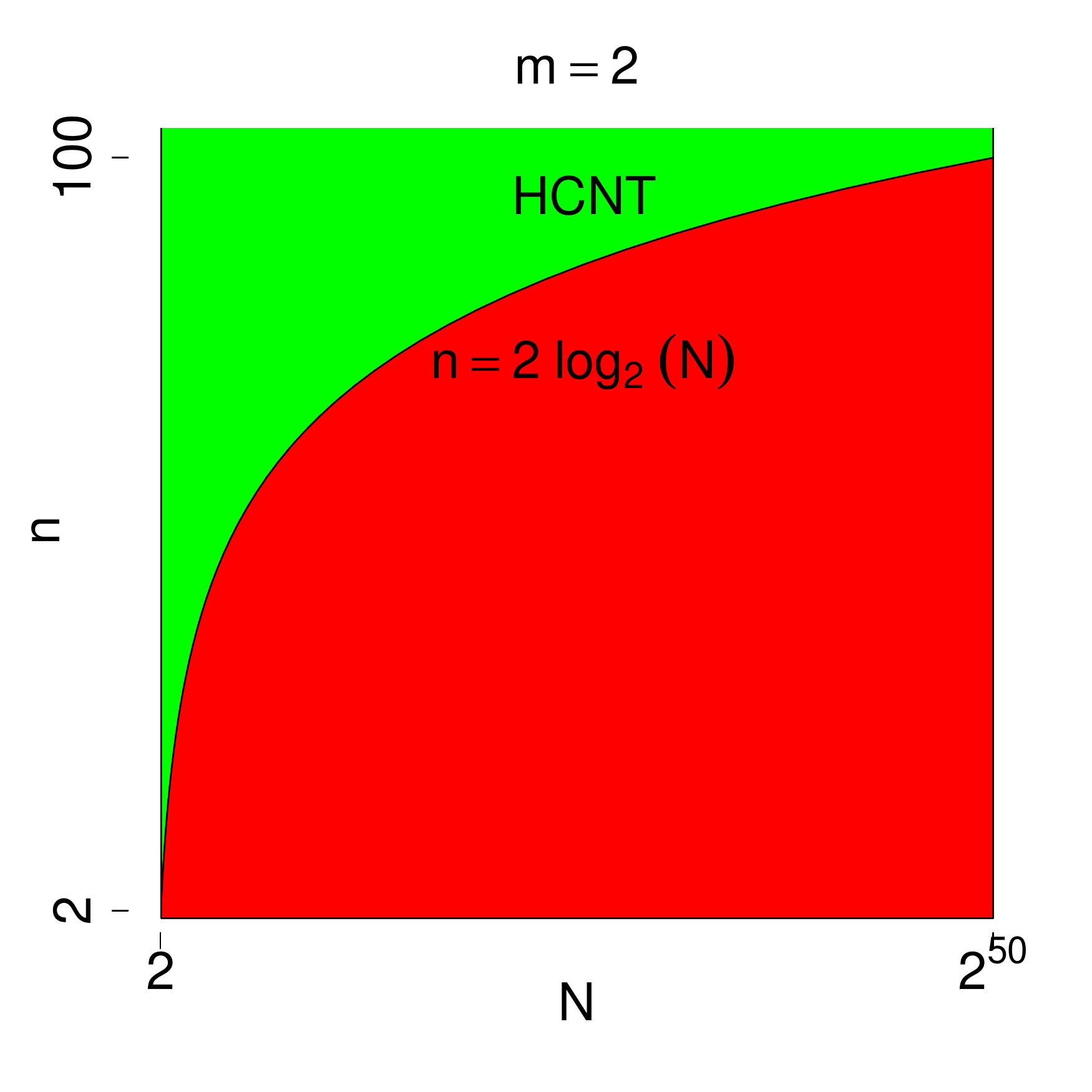}
\hspace{20mm}
\includegraphics[scale=0.4]{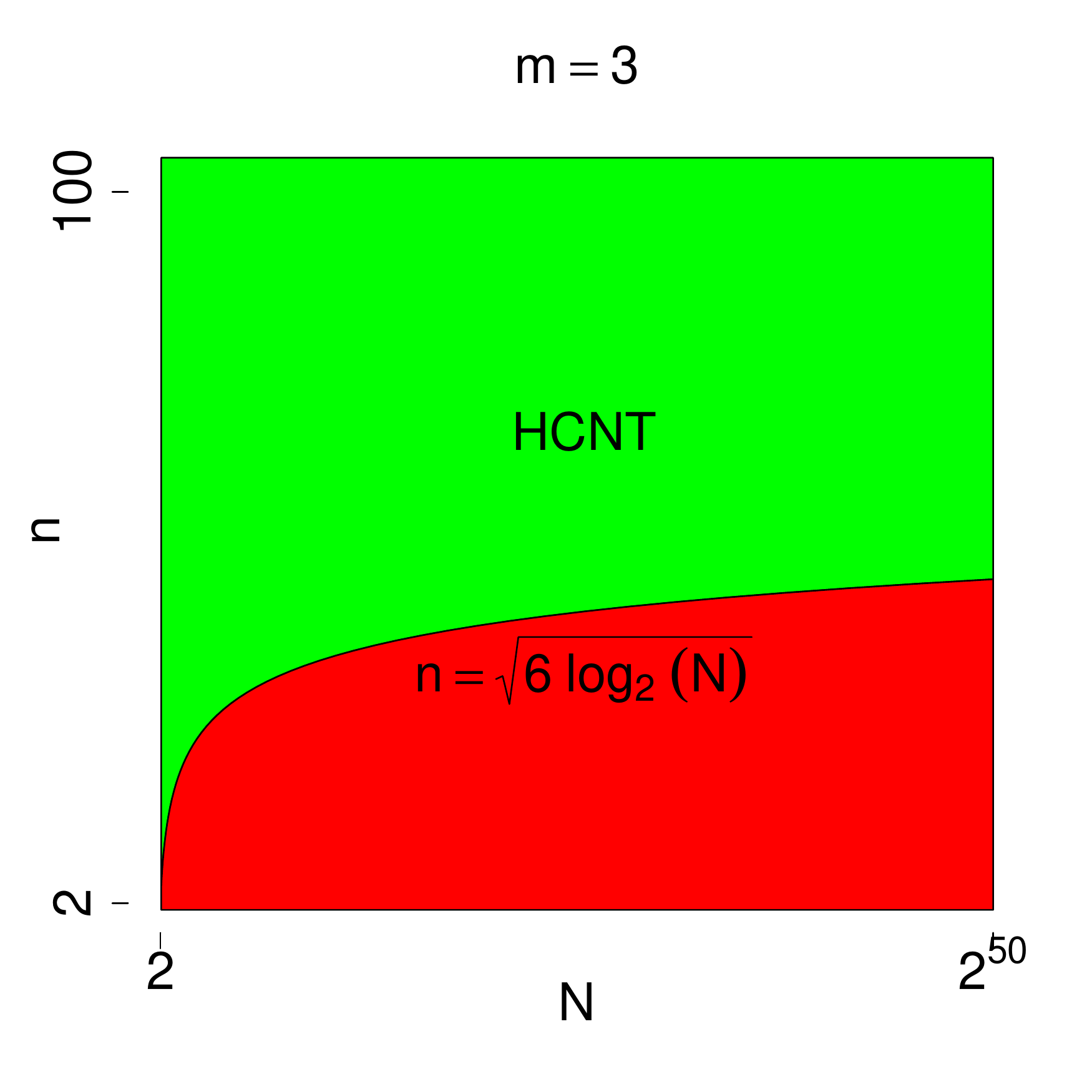}
\vspace{-4mm}
\caption{\it\small Detection boundary in $(N,n)$ for testing PC ($m=2$) and HPC ($m=3$) when $p_0=1/2$. Red: undetectable; green: detectable.}
\label{HPCphase}
\vspace{-5mm}
\end{figure}

\section{Detection of general dense subhypergraph with known $p_0$ and $p_1$}\label{subgraph}
Consider problem (\ref{hypothesis}) with known $p_0$ and $p_1$ satisfying $0<p_0<p_1<1$. Moreover, assume that 
\begin{equation}\label{cond:1}
 \frac{\log N}{n^{m-1}}=o(1), \hskip 1cm \log\left(1\vee \frac{1}{n^{m-1}p_0}\right)=o\left(\log\frac{N}{n}\right).
\end{equation}
Condition (\ref{cond:1}) requires $p_0$ and $p_1$ being suitably away from zero and $n$ being suitably large compared with $N$.
The following theorem provides a lower detection boundary characterized by $n,N,m,p_0,p_1$.
\begin{Theorem}\label{thm:1}
Let $m\geq 2$ and (\ref{cond:1}) hold. 
All tests are asymptotically powerless if
\begin{equation}\label{cond:2}
\frac{p_1-p_0}{\sqrt{p_0}}\left(\frac{n^2}{N}\right)^{\frac{m}{2}}=o(1),
\end{equation}
\begin{equation}\label{cond:3}
\limsup_{n,N\to\infty} \frac{(n-1)(n-2)\cdots(n-m+1)H_{p_0}(p_1)}{m!\log\frac{N}{n}}<1.
\end{equation}
\end{Theorem}
Proof of Theorem \ref{thm:1} proceeds by showing that 
The lower detection boundary in Theorem \ref{thm:1} is actually sharp, as revealed by the following theorem.
Define the hypergraphic total degree test as
\[
\text{(HTDT):}\,\,\,\, W=\sum_{1\leq i_1<\ldots<i_m\leq N}A_{i_1i_2\ldots i_m},
\]
and the hypergraphic scan test as
\[
\text{(HST):}\,\,\,\, W_n=\max_{|S|=n,S\subset \mathcal{V}}W_S,
\]
where $W_S=\sum_{\substack{1\leq i_1<\ldots<i_m\leq N\\
i_1,\ldots,i_m\in S}}A_{i_1i_2\ldots i_m}$ is the number of edges restricted to $S$.
\begin{Theorem}\label{thm:2}
Let $m\ge2$ and (\ref{cond:1}) hold. The following results hold.
\begin{enumerate}[label={(\Roman*)},itemindent=1em]
\item
Suppose $\lim\limits_{N\to\infty}N^mp_0=\infty$. HTDT is asymptotically powerful if
\begin{equation}\label{cond:4}
\frac{p_1-p_0}{\sqrt{p_0}}\left(\frac{n^2}{N}\right)^{\frac{m}{2}}\to\infty.
\end{equation} \label{thm:2:I}
\item
Suppose $\lim\limits_{n\to\infty}n^mp_1=\infty$. 
HST is asymptotically powerful if
\begin{equation}\label{cond:5}
\limsup_{n,N\to\infty} \frac{(n-1)(n-2)\cdots(n-m+1)H_{p_0}(p_1)}{m!\log\frac{N}{n}}>1.
\end{equation}
\label{thm:2:II}
\end{enumerate}
\end{Theorem}
The condition $N^mp_0\to\infty$ in Theorem \ref{thm:2} requires that the average degree of the hypergraph tends to infinity, which is useful to ensure detectability. To further propose
consistent detection algorithms, one typically needs stronger condition such as $N^{m-1}p_0\rightarrow\infty$ (see \cite{GD14,GD17,YLFS20}).
The condition $n^mp_1\to\infty$ requires the underlying subhypergraph being suitably dense.
Given $m,n,N$, Theorems \ref{thm:1} and \ref{thm:2} together provide a region $R_{m}$ of $(p_0,p_1)$ in which the dense subhypergraph is undetectable:
\[
R_{m}=\left\{(p_1,p_0): 0<p_0<p_1<1,\ \frac{p_1-p_0}{\sqrt{p_0}}\ll\left(\frac{N}{n^2}\right)^{\frac{m}{2}},\ H_{p_0}(p_1)<\frac{m!}{n^{m-1}}\log\frac{N}{n}\right\},\ \ m\geq 2.
\]
In particular, $R_{2}$ with $m=2$ degenerates to \cite{AV14}. 
It can be checked that when $n\lesssim\sqrt{N}$, $R_m\subsetneq R_2$ if $m>2$; when $n=o(\sqrt{N})$, such inclusion no longer holds if $m>2$. 
See Figures \ref{subHperphase} and \ref{Hperphase}
an illustration for $R_{2}$ and $R_{3}$.
This may indicate that testing subhypergraph is generally easier than testing subgraph when $n\lesssim\sqrt{N}$, while the two testing problems are generally incomparable when $n=o(\sqrt{N})$.

\begin{figure}[htp] 
\vspace{-3mm}
\centering
\includegraphics[scale=0.4]{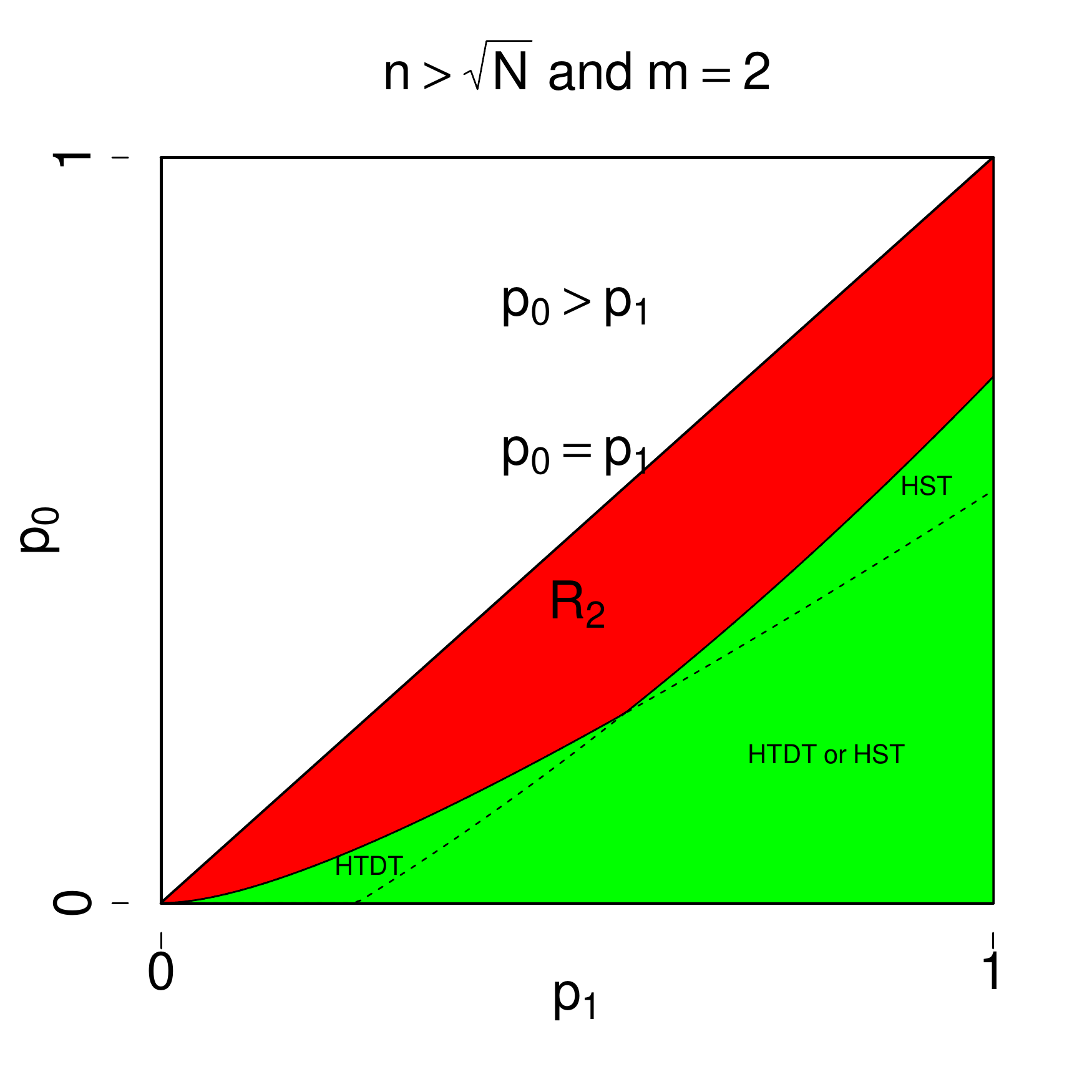}
\hspace{20mm}
\includegraphics[scale=0.4]{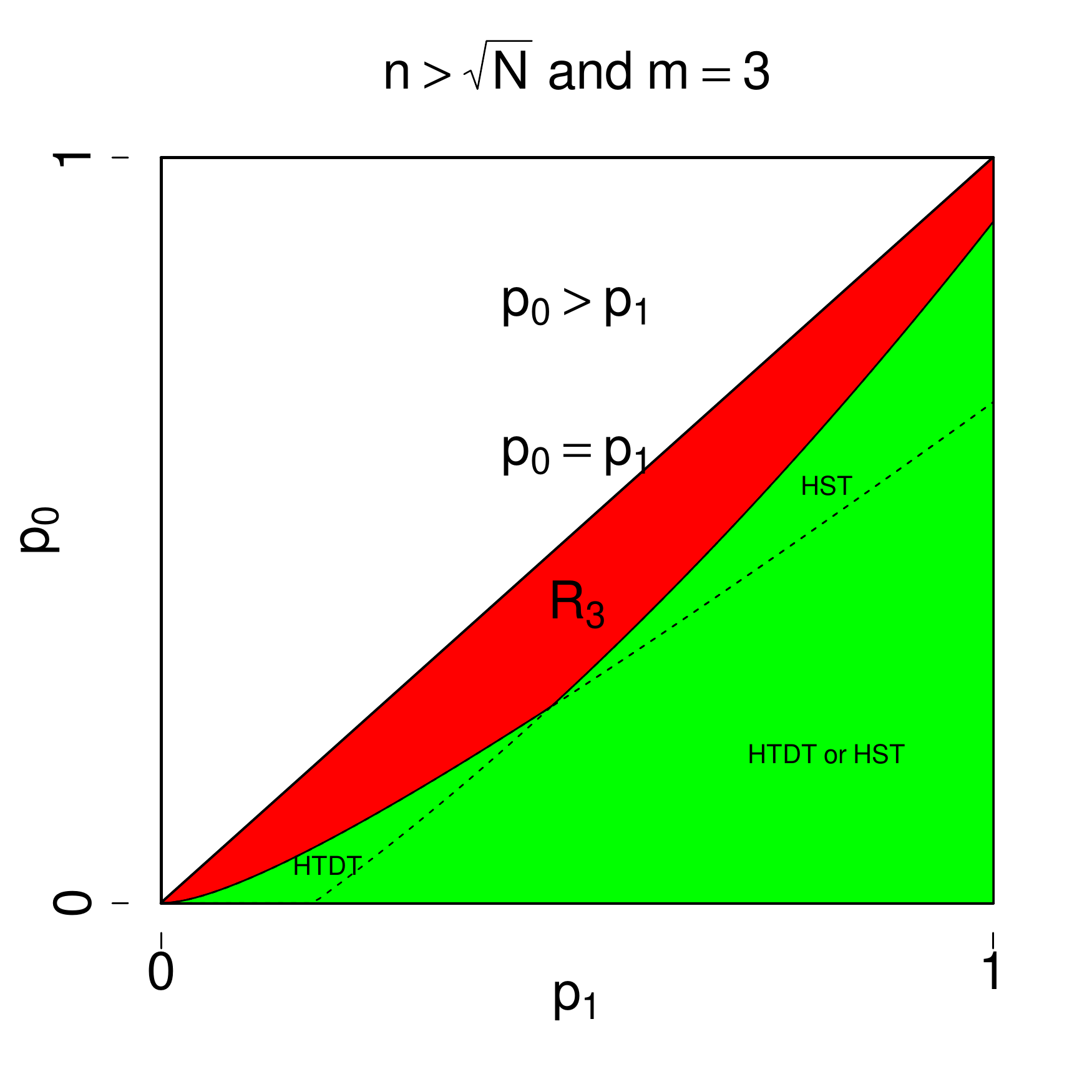}
\vspace{-5mm}
\caption{\it\small Detection boundaries in $(p_1,p_0)$ for testing (\ref{hypothesis}) when $p_0,p_1$ are known and $n\succ\sqrt{N}$. Red: undetectable; green: detectable.}
\label{subHperphase}
\end{figure}

\begin{figure}[htp] 
\vspace{-3mm}
\centering
\includegraphics[scale=0.4]{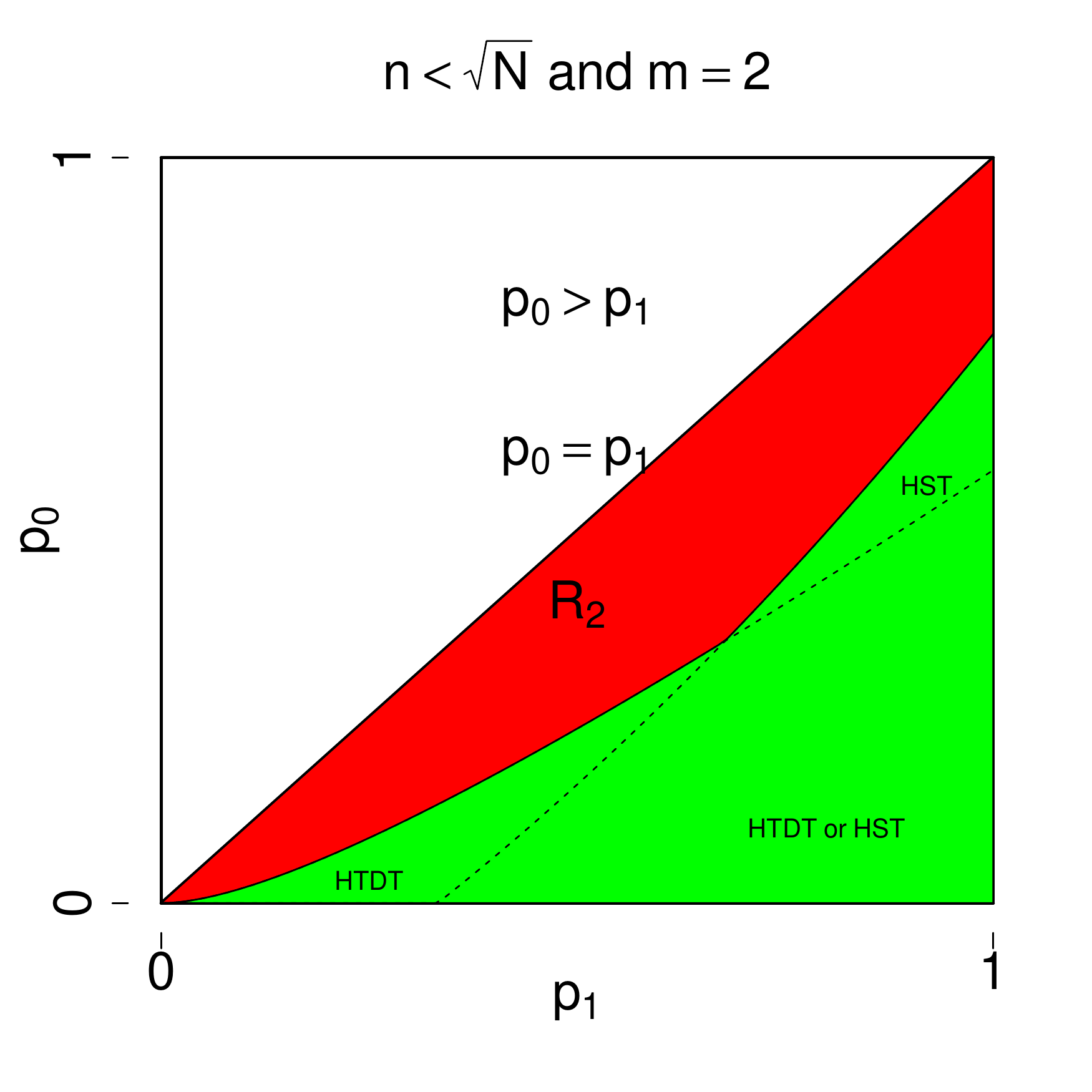}
\hspace{20mm}
\includegraphics[scale=0.4]{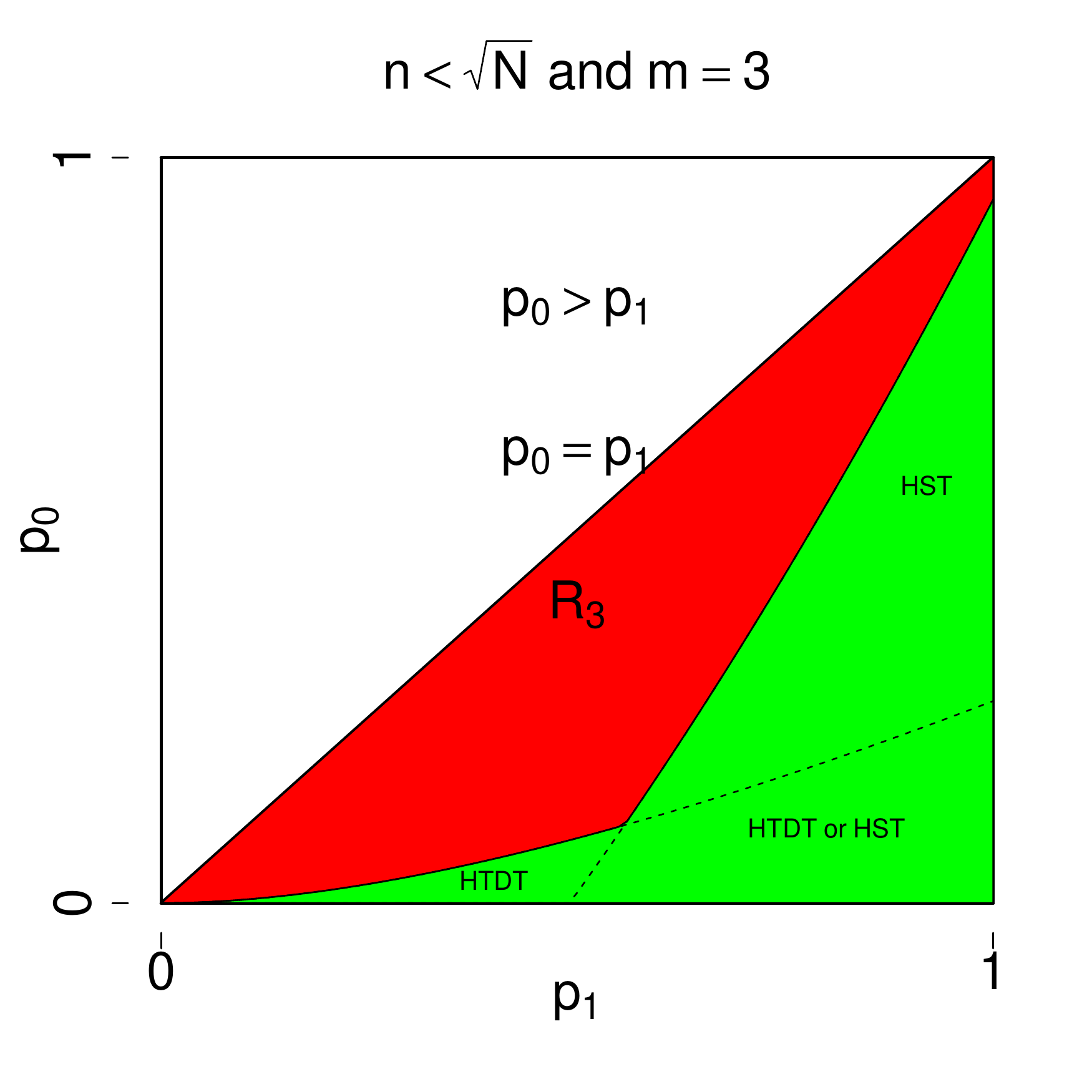}
\vspace{-5mm}
\caption{\it\small Detection boundaries in $(p_1,p_0)$ for testing (\ref{hypothesis}) when $p_0,p_1$ are known and $n=o(\sqrt{N})$. Red: undetectable; green: detectable.}
\label{Hperphase}
\end{figure}

\section{Detection of general dense subhypergraph with unknown $p_0$ and $p_1$}\label{unknown}
Consider the more realistic scenario that both $p_0$ and $p_1$ are unknown
but still satisfy $0<p_0<p_1<1$. 
Now the hypotheses in (\ref{hypothesis}) become composite which poses challenges for directly analyzing $\gamma_N^*(T)$,
defined in (\ref{new:risk}), due to the supremum therein. Inspired by \cite{AV14}, we consider the following auxiliary testing problem:
\begin{equation}\label{unknownp0:hypothesis}
H_0: A\sim \mathcal{H}_m(N,p_0)\,\,\,\,\text{vs.}\,\,\,\, H_1^\prime: A\sim\mathcal{H}_m(N,p_0^{\prime}; n,p_1),
\end{equation}
where $p_0^\prime$ is defined in (\ref{p0:prime}).
It can be checked that the hypergraphs under $H_0$ and $H_1^\prime$ have the same expected total degrees.
We will derive a sharp detection boundary for testing (\ref{hypothesis}) under the following rate conditions:
\begin{equation}\label{ucond:1}
 \frac{\log N}{n^{m-1}}=o(1), \hskip 1cm \log\left(1\vee \frac{1}{n^{m-1}p_0^{\prime}}\right)=o\left(\log\frac{N}{n}\right).
\end{equation}
The following theorems provide a sharp detection boundary characterized 
by $n,N,m,p_0',p_1$. 
\begin{Theorem}\label{thm:4}
Suppose condition (\ref{ucond:1}) holds. All tests are asymptotically powerless if
\begin{equation}\label{cond4:1}
\frac{p_1-p_0^{\prime}}{\sqrt{p_0^{\prime}}}\left(\frac{n^2}{N}\right)^{\frac{m+1}{4}}=o(1),
\end{equation}
\begin{equation}\label{cond4:2}
\limsup\limits_{n,N\to\infty} \frac{(n-1)(n-2)\cdots(n-m+1)H_{p_0^{\prime}}(p_1)}{m!\log\frac{N}{n}}<1.
\end{equation}
\end{Theorem}
The proof of Theorem \ref{thm:4} proceeds as follows.
Under (\ref{cond4:1}) and (\ref{cond4:2}), it holds that for any test $T$,
$\gamma'_N(T)\to1$, where $\gamma_N'(T)=\mathbb{P}_0(T=1)+\max_{|S|=n}\mathbb{P}'_1(T=0)$ is the risk of $T$ for testing (\ref{unknownp0:hypothesis}).
Then $\gamma_N^*(T)\to1$ due to the trivial fact $\gamma_N^*(T)\ge\gamma'_N(T)$.
The result of Theorem \ref{thm:4} is sharp in the sense that there exist asymptotically powerful tests (in the sense of $\gamma_N^*$ risk)
if either (\ref{cond4:1}) or (\ref{cond4:2}) does not hold, as revealed in the following Theorem \ref{thm:6}.
Before stating our results, we define two new tests, namely,
hypergrahic loose 2-path test
(HL2-PT) and hypergraphic tight 2-path test (HT2-PT).
Following \cite{P16}, two hypergraphic edges $e_1=\{i_1, i_2,\ldots, i_m\}$ and $e_2=\{j_1, j_2, \ldots, j_m\}$ form a loose 2-path if $|e_1\cap e_2|=1$;
form a tight 2-path if $|e_1\cap e_2|=m-1$. When $m=2$, a loose 2-path is identical to a tight 2-path. But they are different when $m\ge3$.
Let
\[
V_1=(m-1)!(N-1)^{(m-1)}\frac{N^{(m)}}{N^{(m)}-1}\hat{p}_0(1-\hat{p}_0),\,\,\,
V_2=\frac{1}{N-m!}\sum_{i_1=1}^N\left[W_{i_1*}-(m-1)!(N-1)^{(m-1)}\hat{p}_0\right]^2,
\]
where
\[
\hat{p}_0=\frac{\sum_{1\leq i_1<i_2<\ldots<i_m\leq N}A_{i_1i_2\ldots i_m}}{N^{(m)}} ,\hskip 1cm W_{i_1*}=\sum_{\text{$i_2, i_3,\ldots,i_m$ are pairwise distinct}}A_{i_1i_2\ldots i_m}.
\]
For $m\ge 2$, define HL2-PT as
\[
\text{(HL2-PT):}\,\,\,\,\mathcal{T}_1=\frac{V_2-V_1}{N^{\frac{2m-3}{2}}\hat{p}_0},
\]
and for $m\ge3$, define HT2-PT as
\[ 
\text{(HT2-PT):}\,\,\,\,
\mathcal{T}_2=\frac{\sum_{\text{$i_1,\ldots,i_{m+1}$
are pairwise distinct}}(A_{i_1\ldots i_m}-\hat{p}_0)(A_{i_2\ldots i_{m+1}}-\hat{p}_0)}{\sqrt{(m+1)!N^{(m+1)}\hat{p}_0^2(1-\hat{p}_0)^2}}.
\]
Note that $V_2$ contains a term $W_{i_1*}^2$ which is actually the number of loose 2-path, so $\mathcal{T}_1$ is called as HL2-PT.
The numerator of $\mathcal{T}_2$ contains a sum of $A_{i_1i_2\ldots i_m}A_{i_2\ldots i_m i_{m+1}}$ which is actually the number of tight 2-path, so $\mathcal{T}_2$ is called as HT2-PT.
\begin{Theorem}\label{thm:6}
Let $m\ge2$ and condition (\ref{ucond:1}) hold. The following results hold. 
\begin{enumerate}[label={(\Roman*)},itemindent=1em]
\item Suppose $p_0N^{m}>2n$. HST is asymptotically powerful if 
\begin{equation}\label{cond6:1}
\lim\sup \frac{(n-1)(n-2)\cdots(n-m+1)H_{p_0^{\prime}}(p_1)}{m!\log\frac{N}{n}}>1.
\end{equation}\label{thm:6:I}

\item  Suppose $\lim\inf N^{m-1}p_0>1$.
HL2-PT is asymptotically powerful if 
\begin{equation}\label{cond7:1}
n\succ\sqrt{N}\,\,\,\,\text{and}\,\,\,\,\frac{p_1-p_0^{\prime}}{\sqrt{p_0^{\prime}}}\left(\frac{n^2}{N}\right)^{\frac{m+1}{4}}\rightarrow\infty.
\end{equation}\label{thm:6:II}
\item
HT2-PT is asymptotically powerful if 
\begin{equation}\label{thm8eq}
n=o(\sqrt{N})\,\,\,\,\text{and}\,\,\,\,
\frac{p_1-p_0^{\prime}}{\sqrt{p_0^{\prime}}}\left(\frac{n^2}{N}\right)^{\frac{m+1}{4}}\rightarrow\infty.
\end{equation}\label{thm:6:III} 
\end{enumerate}
\end{Theorem}
When $m=2$, HL2-PT degenerates to the degree variance test considered in \cite{AV14} which is proven asymptotically powerful if $\frac{p_1-p_0^{\prime}}{\sqrt{p_0^{\prime}}}\left(\frac{n^2}{N}\right)^{\frac{3}{4}}\rightarrow\infty$; see Proposition 5 therein. 
When $m\ge3$, we are not able to show the asymptotic powerfulness (in the sense of $\gamma_N^*$ risk) of HL2-PT 
merely under the condition $\frac{p_1-p_0^{\prime}}{\sqrt{p_0^{\prime}}}\left(\frac{n^2}{N}\right)^{\frac{m+1}{4}}\rightarrow\infty$. We need an additional condition $n\succ\sqrt{N}$ to achieve such goal as in (\ref{cond7:1}).
When $n=o(\sqrt{N})$, we propose a new test HT2-PT which is proven asymptotically powerful as in (\ref{thm8eq}). 
This reveals a substantial difference in testing the hypergraphic hypothesis (\ref{unknownp0:hypothesis}) versus testing the graphic counterpart \cite{AV14}.
It is worth mentioning that the result of in Theorem \ref{thm:6} \ref{thm:6:II} can be further extended, i.e., HL2-PT is asymptotically powerful if
\begin{equation}\label{cond7:1*}
\frac{p_1-p_0^{\prime}}{\sqrt{p_0^{\prime}}}\left(\frac{n^2}{N}\right)^{\frac{m+1}{4}}\to\infty\,\,\,\,\text{and}\,\,\,\, \frac{p_1-p_0^{\prime}}{\sqrt{p_0^{\prime}}}\left(\frac{n^2}{N}\right)^{\frac{2m-1}{4}}\rightarrow\infty.
\end{equation}
The sharp detection boundaries stated in Theorem \ref{thm:6}
are displayed in Figure \ref{unknownsubHperphase} and Figure \ref{unknownHperphase},
in which the undetectable region of $(p_1,p_0')$ is defined by
\[
R_{m}=\left\{(p_1,p_0^{\prime}): 0<p_0^{\prime}<p_1<1,\ \frac{p_1-p_0^{\prime}}{\sqrt{p_0^{\prime}}}\ll\left(\frac{N}{n^2}\right)^{\frac{m+1}{4}},\ H_{p_0^{\prime}}(p_1)<\frac{m!}{n^{m-1}}\log\frac{N}{n}\right\},\ \ m\geq 2.
\]
It can be seen that, when $n\succ\sqrt{N}$, $R_m\subsetneq R_2$ if $m>2$;
when $n=o(\sqrt{N})$, such inclusion no longer holds if $m>2$.
This phenomenon is similar to Figures \ref{subHperphase} and \ref{Hperphase}. 
The difference is that we need HST or HL2-PT in the former
while need HST or HT2-PT in the latter.

\begin{figure}[htp] 
\vspace{-3mm}
\centering
\includegraphics[scale=0.4]{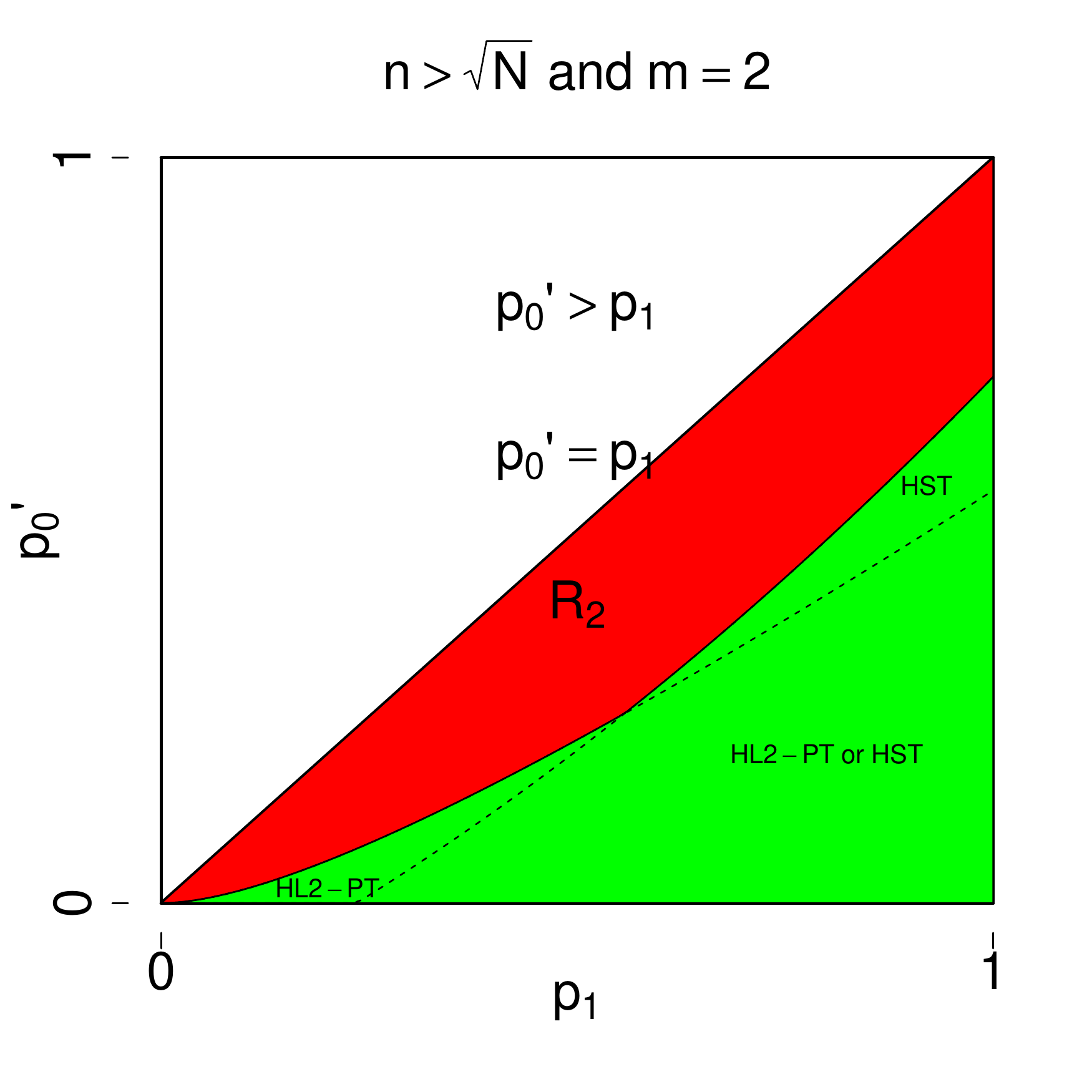}
\hspace{20mm}
\includegraphics[scale=0.4]{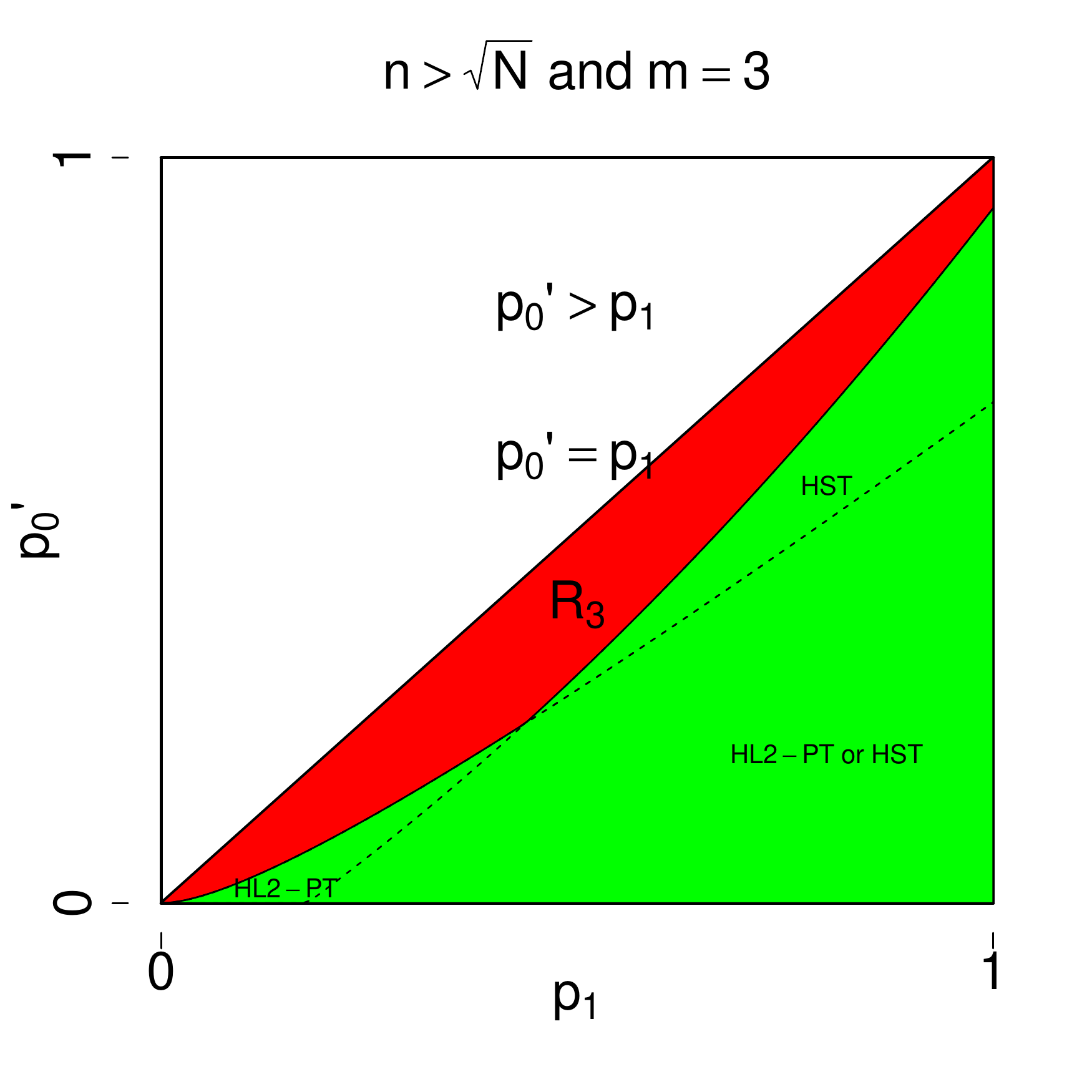}
\vspace{-5mm}
\caption{\it\small Detection boundaries in $(p_1,p_0^{\prime})$ for testing (\ref{hypothesis}) when $p_0,p_1$ are unknown and $n\succ\sqrt{N}$. Red: undetectable; green: detectable.}
\label{unknownsubHperphase}
\end{figure}
\begin{figure}[htp] 
\vspace{-3mm}
\centering
\includegraphics[scale=0.4]{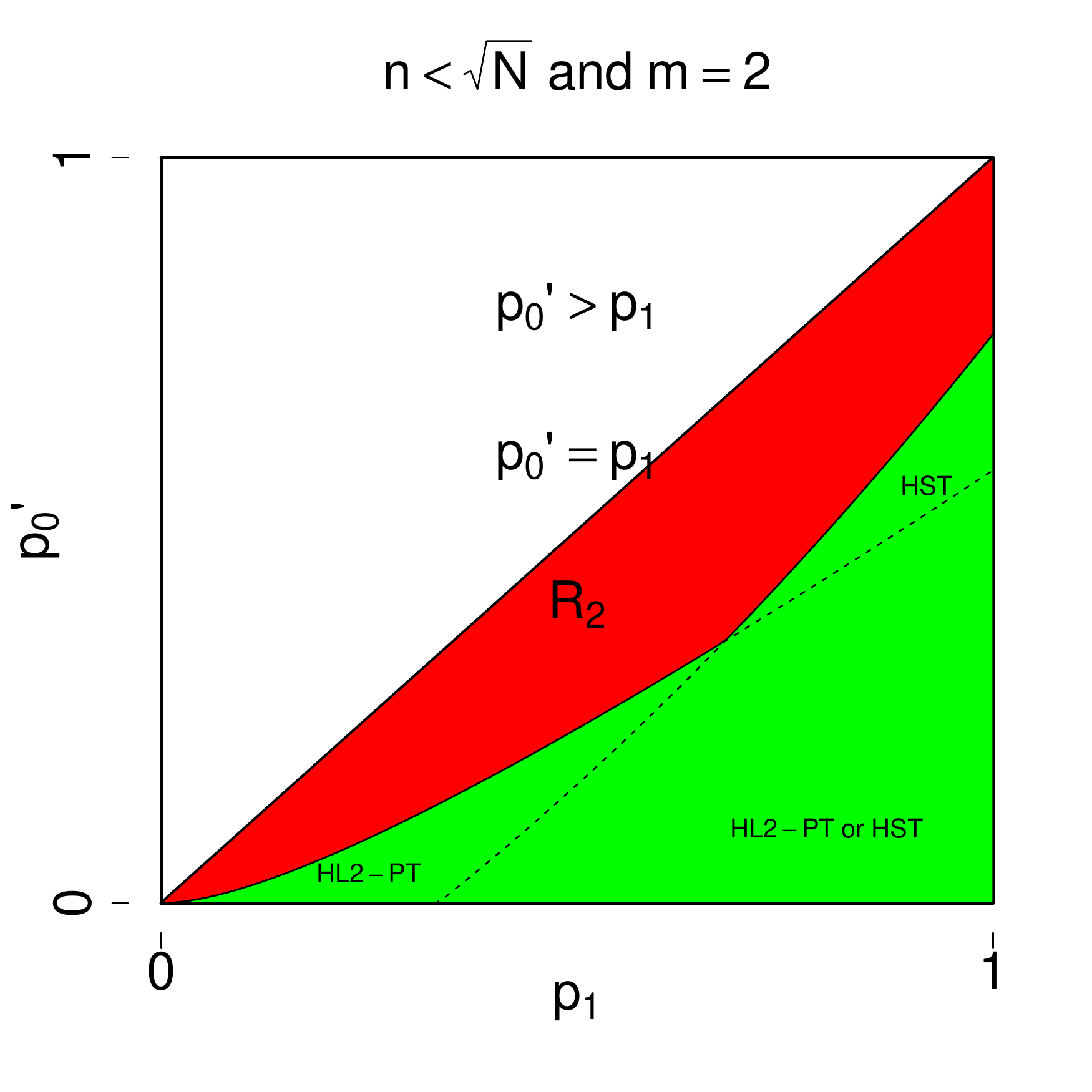}
\hspace{20mm}
\includegraphics[scale=0.4]{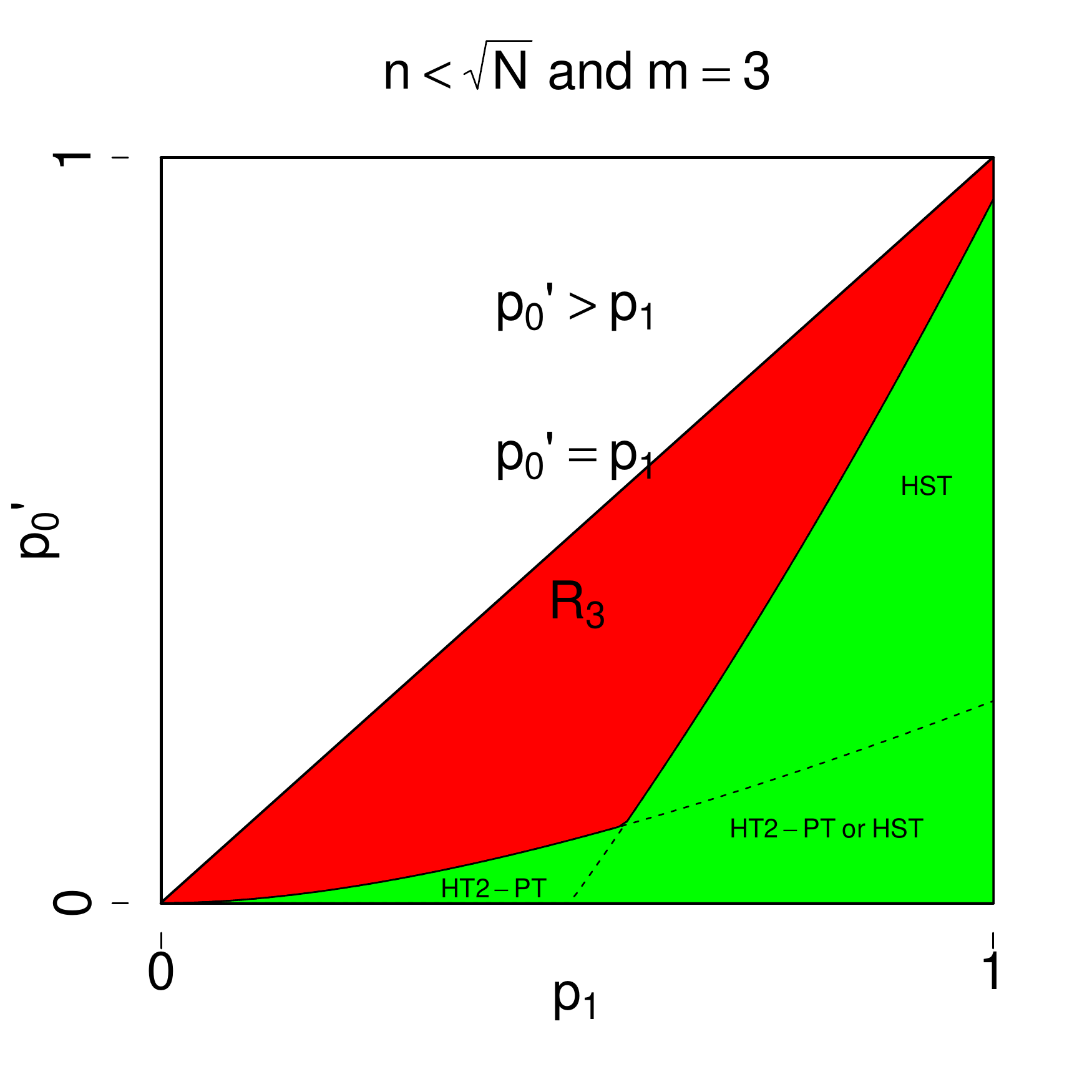}
\vspace{-5mm}
\caption{\it\small Detection boundaries in $(p_1,p_0^{\prime})$ for testing (\ref{hypothesis}) when $p_0,p_1$ are unknown and $n=o(\sqrt{N})$. Red: undetectable; green: detectable.}
\label{unknownHperphase}
\end{figure}




\section{Proofs.}\label{proof}
In this section, we prove the main results of this paper.
\begin{proof}[Proof of Theorem \ref{clique0} \ref{thm:1:I}] 
By direct examinations,
we can show that the likelihood ratio statistic for testing (\ref{hypothesis}) is equal to
\[
L=\frac{|\{S\subset [N]: |S|=n, W_S=n^{(m)}\}|}{N^{(n)}p_0^{n^{(m)}}}.
\]
Let $S_1,S_2$ be two independent and uniformly selected subsets of $[N]$ with cardinality $n$, and let $K=|S_1\cap S_2|$. 
By (37) in \cite{AV14}, it follows that
\[\mathbb{P}(K\geq k)\leq \exp\left[-nH_{\rho}(\frac{k}{n})\right]=\exp\left[-k\left(\log\frac{kN}{n^2}+O(1)\right)\right],
\]
where $\rho=\frac{n}{N-n}$.
Then we have
\begin{eqnarray}\nonumber
\mathbb{E}_0L^2&=&\mathbb{E}_0p_0^{-K^{(m)}}\leq \mathbb{P}_0(K\leq m-1)+\sum_{k=m}^np_0^{-k^{(m)}}\exp\left(-nH_{\rho}(\frac{k}{n})\right)\\  \label{cliqeq1}
&\leq&1+\sum_{k=m}^n\exp\left(-k\left[\log\frac{kN}{n^2}-\frac{k^{(m)}}{k}\log\frac{1}{p_0}+O(1)\right]\right).
\end{eqnarray}

Let $f(k)=\log\frac{kN}{n^2}-\frac{k^{(m)}}{k}\log\frac{1}{p_0}$, for $k\ge1$.
By simple calculations, it can be shown that 
$f(k)\geq c_n:=\min\{f(m),f(n)\}$. Note that the condition in part (\ref{thm:1:I}) implies $n=o(\log N)$, we get that
\begin{eqnarray*}
f(n)&=&\log\frac{N}{n}-\frac{n^{(m)}}{n}\log\frac{1}{p_0}
\geq\log N-\frac{n^{m-1}}{m!}\log\frac{1}{p_0}-\log n\\
&\geq& \log N-(1-\epsilon)\log N-\log n
\rightarrow\infty.
\end{eqnarray*}
Besides,
\[f(m)=\log\frac{mN}{n^2}-\frac{1}{m}\log\frac{1}{p_0}=\log m+f(n)+\left(\frac{n^{(m)}}{n}-\frac{1}{m}\right)\log\frac{1}{p_0}-\log n\rightarrow\infty.\]
Hence $c_n\rightarrow\infty$.
Then by (\ref{cliqeq1}), we conclude that
\[\mathbb{E}_0L^2\leq 1+\sum_{k=m}^n\exp(-kc_n)=1+o(1).\]
As a result, $L$ converges to 1 in distribution under $H_0$.
\end{proof}

\begin{proof}[Proof of Theorem \ref{clique0} \ref{thm:1:II}]
Note that if $\omega_N\geq n$, then there exists at least one clique
of cardinality $n$ in the hypergraph by the definition of clique number \cite{BP09}.
Let $X_n$ denote the number of such cliques in an $m$-uniform hypergraph. Under $H_0$, we have
\begin{eqnarray*}
\mathbb{P}(\omega_N\geq n)\leq\mathbb{P}(X_n\geq 1)\leq \mathbb{E}X_n=\binom{N}{n}p_0^{\binom{n}{m}}.
\end{eqnarray*}
Under the condition of part \ref{thm:1:II}, 
by Stirling's approximation, one has
\begin{eqnarray*}
\log \binom{N}{n}+\binom{n}{m}\log p_0&\leq& n\log N-n\log n+n+O(\log n)-\frac{n^m}{m!(1+\frac{\epsilon}{2})}\log \frac{1}{p_0}\\ 
&\leq& n\left[\log N-\frac{1+\epsilon}{1+\frac{\epsilon}{2}}\log N-\log n+1+o(1)\right]\rightarrow-\infty.
\end{eqnarray*}
Then we have $\mathbb{P}(\omega_N\leq n-1)=1-o(1)$ under $H_0$. 
Under $H_1$, there exists at least one clique of cardinality $n$ and hence $\omega_N\geq n$. Consequently, the clique number test is asymptotically powerful.
\end{proof}

Let's provide several lemmas before proving Theorem \ref{thm:1}. 
\begin{Lemma}[Lemma 3 of \cite{AV14}]\label{lem:0}
For $p\in (0,1)$, $H_p(q)$ is convex and increasing in $q\in [0,1]$. Moreover,
  \[
    H_p(q)=\left\{
                \begin{array}{ll}
                  \frac{(q-p)^2}{2p(1-p)}+O\left(\frac{(q-p)^3}{p^2}\right), \hskip 1cm \frac{q}{p}\rightarrow1;\\
                  p(r\log r-r+1), \hskip 1cm \frac{q}{p}\rightarrow r\in(1,\infty),p\rightarrow0;\\
                  q\log\frac{q}{p}+O(q), \hskip 1cm \frac{q}{p}\rightarrow \infty.
                \end{array}
              \right.
  \]
\end{Lemma}

Additional lemmas are presented below. Define
\[\Delta=\log\left(1+\frac{(p_1-p_0)^2}{p_0(1-p_0)}\right),\hskip 1cm \ k_*=\left(1+\frac{m!\log\frac{N}{n}}{\Delta}\right)^{\frac{1}{m-1}}\wedge n,\]
\[k_{min}=\left(1+\frac{m!\log\frac{Nk_*}{n^2}-\log\left(\log\frac{n^{m-1}}{\log\frac{N}{n}}\wedge\log\frac{N}{n}\right)}{\Delta}\right)^{\frac{1}{m-1}}\wedge n.\]

\begin{Lemma}\label{lem:1}
Under conditions (\ref{cond:1}), (\ref{cond:2}),  (\ref{cond:3}), we have $k^*\sim k_{min}\rightarrow\infty$ and $\log\frac{n}{k_{min}}=o\left(\log \frac{N}{n}\right)$.
\end{Lemma}

\begin{proof}[Proof of Lemma \ref{lem:1}] 
It suffices to prove $k_*\rightarrow\infty$ and $\log\frac{n}{k_{*}}=o\left(\log \frac{N}{n}\right)$.
Suppose $\frac{p_1}{p_0}\rightarrow1$, then $\Delta\sim 2H_{p_0}(p_1)$ by Lemma \ref{lem:0}. Hence, we have
\[k_*\succ\left(\frac{\log\frac{N}{n}}{H_{p_0}(p_1)}\right)^{\frac{1}{m-1}}\wedge n\succ\left(\frac{n^{m-1}\log\frac{N}{n}}{m!\log\frac{N}{n}}\right)^{\frac{1}{m-1}}\wedge n\succ n\rightarrow\infty,\]
from which it follows that $\log\frac{n}{k_*}=O(1)$.

If $\frac{p_1}{p_0}\rightarrow r\in(1,\infty)$, then $H_{p_0}(p_1)=o(1)$ by condition (\ref{cond:1}) and $p_0=o(1)$. By Lemma \ref{lem:0} and condition (\ref{cond:3}), we get that
$H_{p_0}(p_1)\sim p_0(r\log r-r+1)\prec \frac{\log\frac{N}{n}}{n^{m-1}}$. We therefore conclude that $n^{m-1}\prec \frac{\log\frac{N}{n}}{p_0}$ and
$k_*\succ\left(\frac{\log\frac{N}{n}}{p_0(r-1)^2}\right)^{\frac{1}{m-1}}\wedge n\succ n\rightarrow\infty$.
Consequently $\log\frac{n}{k_*}=O(1)$.

If $\frac{p_1}{p_0}\rightarrow \infty$, then $p_0=o(1)$ and  by Lemma \ref{lem:0} and condition (\ref{cond:3}), we get that
$H_{p_0}(p_1)=p_1\log\frac{p_1}{p_0}+O(p_1)\prec  \frac{\log\frac{N}{n}}{n^{m-1}}$.
As a result,
\begin{equation}\label{lem1:1}
\frac{p_1}{p_0}\log\frac{p_1}{p_0}\prec \frac{\log\frac{N}{n}}{n^{m-1}p_0}.
\end{equation}
By condition (\ref{cond:1}), it follows that
$\log\frac{p_1}{p_0}\prec \log \frac{\log\frac{N}{n}}{n^{m-1}p_0}=o\left(\log \frac{N}{n}\right)$.
By direct caculations, it holds that $1+\frac{(p_1-p_0)^2}{p_0(1-p_0)}-\frac{p_1}{p_0}=\frac{(p_1-p_0)(p_1-1)}{p_0(1-p_0)}\leq 0$.
Therefore, we have
$k_*\succ\left(\frac{\log\frac{N}{n}}{\log\frac{p_1}{p_0}}\right)^{\frac{1}{m-1}}\rightarrow\infty$.

By (\ref{lem1:1}), $p_1\prec \frac{\log\frac{N}{n}}{n^{m-1}}$. Then 
\begin{eqnarray}
\left(\frac{n}{k_*}\right)^{m-1}\prec\left(\frac{n}{\left[\frac{\log \frac{N}{n}}{\log(1+\frac{p_1^2}{p_0})}\right]^{\frac{1}{m-1}}}\right)^{m-1}\prec \frac{n^{m-1}p_1^2}{p_0\log\frac{N}{n}}\prec\frac{n^{m-1}}{p_0\log\frac{N}{n}}\frac{\log^2\frac{N}{n}}{n^{2(m-1)}}=\frac{\log\frac{N}{n}}{n^{m-1}p_0}.
\label{eq:20}
\end{eqnarray}
Combining (\ref{eq:20}) with condition (\ref{cond:1}), it follows that $\log\frac{n}{k_{*}}=o\left(\log \frac{N}{n}\right)$.
\end{proof}
Let $\theta_q=\log\frac{q(1-p_0)}{p_0(1-q)}$ for any $q\in(0,1)$, and $\Lambda(\theta)=\log(1-p_0+p_0e^{\theta})$.
\begin{Lemma}\label{lem2}
Under condition (\ref{cond:1}), for each $k\in [k_{min}+1,n]$, there exists $q_k\in (p_0,1)$ such that
\[
\theta_{q_k}\leq 2\theta_{p_1}\,\,\,\,\text{and}\,\,\,\, \frac{(k-1)(k-2)\ldots (k-m+1)}{m!}H_{p_0}(q_k)=\log\frac{N}{k}+2.
\]
\end{Lemma}
\begin{proof}[Proof of Lemma \ref{lem2}] 
Let $\tilde{q}\in(0,1)$ be uniquely determined by
$\frac{\tilde{q}}{1-\tilde{q}}=\frac{p_1^2(1-p_0)}{p_0(1-p_1)^2}$.
It's easy to verify that $\theta_{\tilde{q}}=2\theta_{p_1}$. 
Note that $H_{p_0}(t)$ is continuous and monotone increasing in $t\in(p_0,\tilde{q})$.
Since for $k\geq k_{min}$, we have
\[
\frac{(k-1)(k-2)\cdots (k-m+1)}{m!}H_{p_0}(\tilde{q})\geq \frac{(k_{min}-1)(k_{min}-2)\cdots (k_{min}-m+1)}{m!}H_{p_0}(\tilde{q}),
\]
and
$\log\frac{N}{k_{min}}\geq \log\frac{N}{k}$.
Therefore, it suffices to prove 
\begin{equation}\label{lemma5.3}
\frac{(k_{min}-1)(k_{min}-2)\cdots (k_{min}-m+1)}{m!}H_{p_0}(\tilde{q})\geq \log\frac{N}{k_{min}}+2.
\end{equation}
By Lemma \ref{lem:1}, $\log\frac{n}{k_{min}}=o\left(\log \frac{N}{n}\right)$, hence we get 
\[\log\frac{N}{k_{min}}+2=\log\frac{N}{n}+\log\frac{n}{k_{min}}+2=\log\left(\frac{N}{n}\right)(1+o(1)).\]
Recall that $k^*\sim k_{min}$ by Lemma \ref{lem:1}. Then we only need to prove
$k_*^{m-1}H_{p_0}(\tilde{q})\geq m! (1+\epsilon)\log\frac{N}{n}$
for some constant $\epsilon>0$.

Observe that $k_*<n$. By definition, $k_*\le n$.
If $k_{*}=n$, then $k_{min}+1\sim k_{*}+1=n+1$ by Lemma \ref{lem:1}. Hence there doesn't exist any $k\in [k_{min}+1,n]$ and we have nothing to prove. Then we have $k_*<n$. In this case,
$k_*\geq \left(\frac{m!\log\frac{N}{n}}{\Delta}\right)^{\frac{1}{m-1}}$ which implies
$k_*^{m-1}H_{p_0}(\tilde{q})\geq \frac{m!\log\frac{N}{n}}{\Delta}H_{p_0}(\tilde{q})$.

In the following three regimes: $\frac{p_1}{p_0}\rightarrow1$; $\frac{p_1}{p_0}\rightarrow r\in (1,\infty)$; $\frac{p_1}{p_0}\rightarrow\infty$ with $\frac{p_1^2}{p_0}\rightarrow0$ or $\frac{p_1^2}{p_0}\rightarrow r_2>0$,
it can be shown that 
\begin{equation}\label{eqn:Hp0}
    H_{p_0}(\tilde{q})\geq (1+\epsilon)\Delta.
\end{equation}
The proof of (\ref{eqn:Hp0}) is almost identical to \cite{AV14}, which is omitted.
Consider $\frac{p_1^2}{p_0}\rightarrow\infty$ (hence, $\frac{p_1}{p_0}\rightarrow\infty$).
Recall that $k_*<n$. By definition, we conclude 
$\tilde{q}=1-o(1)$. Then $H_{p_0}(\tilde{q})\sim\log\frac{1}{p_0}$ by Lemma \ref{lem:0}. Moreover, it's easy to verify $\Delta\sim\log\frac{p_1^2}{p_0}$. As a result,
$\frac{H_{p_0}(\tilde{q})}{\Delta}\sim \left(1-\frac{\log p_1^2}{\log p_0}\right)^{-1}$.
If $\frac{\log p_1^2}{\log p_0}\geq \delta>0$ for some $\delta>0$, then (\ref{eqn:Hp0}) holds.
Next we consider $\frac{\log p_1^2}{\log p_0}=o(1)$,
in which we will show (\ref{lemma5.3}). Since $\tilde{q}\geq 1-\frac{p_0}{p_1^2}$, we have
$\frac{H_{p_0}(\tilde{q})}{\Delta}-1\geq \left(2+o(1)\right)\frac{\log \frac{1}{p_1}}{\log \frac{1}{p_0}}$.
Then we get
\begin{eqnarray}\nonumber
&&\frac{(k_{min}-1)(k_{min}-2)\cdots (k_{min}-m+1)}{m!}H_{p_0}(\tilde{q})\\
&\geq &\left[\log\frac{N}{k_{min}}+2\right]\left[1-\frac{2+o(1)+2\log\frac{n}{k_*}+
\log\log\frac{n^{m-1}}{\log\frac{N}{n}}}{\log\frac{N}{n}}\right]
\times \left[1+(2+o(1))\frac{\log\frac{1}{p_1}}{\log\frac{1}{p_0}}\right].
\label{eqn:23}
\end{eqnarray}
To show (\ref{lemma5.3}), it suffices to show that the product of the 
last two items in (\ref{eqn:23}) is greater than or equal to one.
Since $(1-a)(1+b)=1+b(1-a-\frac{a}{b})\geq 1+\epsilon b\geq 1$ if $a=o(b)$ and $a=o(1)$,
it suffices to prove that 
\begin{equation}\label{kstarn}
\frac{\log\frac{n}{k_*}+\log\log\frac{n^{m-1}}{\log\frac{N}{n}}}{\log\frac{N}{n}}\ll\frac{\log\frac{1}{p_1}}{\log\frac{1}{p_0}}.
\end{equation}
By (\ref{lem1:1}), we get that
$\frac{p_1^2}{p_0}\prec\left[\frac{\log\frac{N}{n}}{n^{m-1}}\right]^2\frac{1}{p_0}$.
Hence
\[
2\log\log\frac{N}{n}-\log n^{m-1}+\log \frac{1}{n^{m-1}p_0}=\log\left(\left[\frac{\log\frac{N}{n}}{n^{m-1}}\right]^2\frac{1}{p_0}\right)\succ \log \frac{p_1^2}{p_0}\rightarrow\infty.
\]
Then one has $2\log\log\frac{N}{n}+\log \frac{1}{n^{m-1}p_0}\gg\log n^{m-1}$.
Since
$\frac{\log\left(\log\frac{N}{n}\right)^2+\log(1\vee \frac{1}{n^{m-1}p_0})}{\log\frac{N}{n}}=o(1)$,
we have $\log n^{m-1}=o\left(\log\frac{N}{n}\right)$. Consequently,
$\frac{\log\frac{1}{p_0}}{\log\frac{N}{n}}=\frac{\log n^{m-1}+\log\frac{1}{n^{m-1}p_0}}{\log\frac{N}{n}}=o(1)$.
Hence, to prove (\ref{kstarn}), we only need to show
$\log \frac{n}{k_*}+\log\log\frac{n^{m-1}}{\log\frac{N}{n}}=O(\log \frac{1}{p_1})$.

Note that
\[\log \left[\frac{n}{k_*} \right]<\log \left[\frac{n}{k_*} \right]^{m-1}\prec \log\frac{n^{m-1}}{1+\frac{\log\frac{N}{n}}{\Delta}}\prec \log\frac{n^{m-1}}{\log\frac{N}{n}}+\log \Delta\prec \log\frac{n^{m-1}}{\log\frac{N}{n}}+\log\log \frac{p_1^2}{p_0}.\]
Since $p_1\log \frac{p_1}{p_0}\prec \frac{\log\frac{N}{n}}{n^{m-1}}$ by (\ref{lem1:1}), it's easy to verify that
\[\frac{\log \frac{n}{k_*}+\log\log\frac{n^{m-1}}{\log\frac{N}{n}}}{\log \frac{1}{p_1} }\leq \frac{\log\frac{n^{m-1}}{\log\frac{N}{n}}+\log\log \frac{p_1^2}{p_0}+\log\log\frac{n^{m-1}}{\log\frac{N}{n}}}{\log\frac{n^{m-1}}{\log\frac{N}{n}}+\log\log \frac{p_1}{p_0}}=O(1).\]
Proof is complete.
\end{proof}

\begin{Lemma}\label{lem:3}
$\min_{k\in [k_{min}+1,n]}\left\{\frac{(k-1)(k-2)\cdots (k-m+1)}{m!}H_{p_1}(q_k)-\log\frac{n}{k}\right\}\rightarrow\infty$.
\end{Lemma}

\begin{proof}[Proof of Lemma \ref{lem:3}]
Note that 
$p_1\leq q_k\leq \tilde{q}$.  If $\frac{p_1}{p_0}\rightarrow1$,
we get $\tilde{q}\sim p_0$. Then for some constant $\eta_0\in(0,1)$, we get
\[\frac{(q_k-p_0)^2}{(p_1-p_0)^2}\sim \frac{H_{p_0}(q_k)}{H_{p_0}(p_1)}\geq \frac{n^{m-1}}{(1-\eta_0)(k-1)(k-2)\cdots (k-m+1)}\geq \frac{1}{1-\eta_0},\]
\[H_{p_1}(q_k)\succ H_{p_0}(q_k)[1-\sqrt{1-\eta_0}]^2,\]
and
\[\frac{(k-1)(k-2)\cdots (k-m+1)}{m!}H_{p_1}(q_k)\succ \frac{(k-1)(k-2)\cdots (k-m+1)}{m!}H_{p_0}(q_k)\geq \log\frac{N}{k}+2\gg 1\vee\log\frac{n}{k}.\]

If $\frac{p_1}{p_0}\rightarrow r\in(1,\infty)$, then $H_{p_0}(p_1)\rightarrow0$ implies $p_1=o(1)$. By definition of $\tilde{q}$, we conclude $\tilde{q}=o(1)$. We can derive
\[H_{p_1}(q_k)\succ H_{p_0}(q_k)\geq \frac{m!(\log\frac{N}{k}+2)}{(k-1)(k-2)\cdots (k-m+1)}.\]
Consequently, 
\[\frac{(k-1)(k-2)\cdots (k-m+1)}{m!}H_{p_1}(q_k)\succ \log\frac{N}{k}+2\geq\log\frac{N}{n}\gg 1\vee\log\frac{n}{k}.\]

If $\frac{p_1}{p_0}\rightarrow \infty$, then $H_{p_0}(p_1)\rightarrow0$ implies $p_1=o(1)$.  Note  that $\frac{q_k}{p_0}\geq\frac{p_1}{p_0}\rightarrow\infty$. Then for some small constant $\eta_0\in(0,1)$,
\begin{eqnarray}\nonumber
&&\frac{q_k}{p_1}\left(1+\frac{\log\frac{q_k}{p_1}}{\log\frac{p_1}{p_0}}\right)\sim\frac{H_{p_0}(q_k)}{H_{p_0}(p_1)}\geq \frac{n^{m-1}H_{p_0}(q_k)}{m!(1-\eta_0)\log\frac{N}{n}}
=\frac{n^{m-1}}{(1-\eta_0)(k-1)(k-2)\cdots (k-m+1)}\frac{\log\frac{N}{n}+2}{\log\frac{N}{n}}\\   \nonumber
&\geq& \frac{n^{m-1}}{(1-\eta_0)(k-1)(k-2)\cdots (k-m+1)}.
\end{eqnarray}
By definition of $\tilde{q}$, we get $1\leq q_k/p_1\leq \tilde{q}/p_1\leq p_1/p_0$. Hence, $\frac{q_k}{p_1}>\frac{n^{m-1}}{2(k-1)(k-2)\cdots (k-m+1)}$. 
Note that $\frac{q_k}{p_0}\leq \frac{p_1^2}{p_0^2}$. Besides, $\log\frac{p_1}{p_0}=o\left(\log \frac{N}{n}\right)$ by the proof of Lemma \ref{lem:1}. If $k=o(n)$, then $\frac{q_k}{p_1}\rightarrow \infty$. In this case, we have
\begin{eqnarray}\nonumber
&&\frac{(k-1)(k-2)\cdots (k-m+1)}{m!}H_{p_1}(q_k)\\ \label{hpqk}
&=&\frac{H_{p_1}(q_k)}{H_{p_0}(p_k)}\frac{(k-1)(k-2)\cdots (k-m+1)}{m!}H_{p_0}(p_k)\\ \nonumber
&=&\frac{H_{p_1}(q_k)}{H_{p_0}(p_k)}\left(\log\frac{N}{k}+2\right)\geq \frac{q_k\log\frac{q_k}{p_1}-q_k+(1-q_k)p_1}{q_k\log\frac{q_k}{p_0} }\log\frac{N}{n}(1+o(1))\\  \nonumber
&\succ&\frac{\log\frac{N}{n}}{\log\frac{p_1}{p_0}}\left(\log\frac{q_k}{p_1}-1+\frac{p_1}{p_k}-p_1\right)\succ\frac{\log\frac{N}{n}}{\log\frac{p_1}{p_0}}\log\frac{n}{k}\gg\log\frac{n}{k}.
\end{eqnarray}
Suppose $n\asymp k$. If $\frac{q_k}{p_1}\rightarrow \infty$, the proof is the same as (\ref{hpqk}). If $\frac{q_k}{p_1}=O(1)$, then $H_{p_1}(q_k)/q_k\succ1$. In this case, similar to (\ref{hpqk}), we have
\[
\frac{(k-1)(k-2)\cdots (k-m+1)}{m!}H_{p_1}(q_k)
\gg1.
\]
\end{proof}

\begin{proof}[Proof of Theorem \ref{thm:1}] For a given subset $S\subset [N]$ with $|S|=n$, the likelihood ratio is equal to
\begin{eqnarray}\nonumber
L_S&=&\frac{\prod_{\substack{i_1<\cdots<i_m\\ \{i_1,\ldots, i_m\}\subset S}}p_1^{A_{i_1\ldots i_m}}(1-p_1)^{1-A_{i_1\ldots i_m}}\prod_{\substack{i_1<\cdots<i_m\\ \{i_1,\ldots, i_m\}\not\subset S}}p_0^{A_{i_1\ldots i_m}}(1-p_0)^{1-A_{i_1\ldots i_m}}}{\prod_{\substack{i_1<\cdots<i_m\\ \{i_1,\ldots, i_m\}\subset S}}p_0^{A_{i_1\ldots i_m}}(1-p_0)^{1-A_{i_1\ldots i_m}}\prod_{\substack{i_1<\cdots<i_m\\ \{i_1,\ldots, i_m\}\not\subset S}}p_0^{A_{i_1\ldots i_m}}(1-p_0)^{1-A_{i_1\ldots i_m}}}\\ \nonumber
&=&\left(\frac{p_1(1-p_0)}{p_0(1-p_1)}\right)^{A_S}\left(\frac{1-p_1}{1-p_0}\right)^{n^{(m)}}=e^{(\theta_{p_1}A_S-n^{(m)}\Lambda(\theta_{p_1}))},
\end{eqnarray}
where
\[\theta_{p_1}=\log\frac{p_1(1-p_0)}{p_0(1-p_1)},\hskip 1cm \Lambda(\theta_{p_1})=\log (1-p_0+p_0e^{p_1}).\]
Then the unconditional likelihood ratio statistic is expressed as
\[L=\binom{N}{n}^{-1}\sum_{|S|=n}L_S.\]
We truncate the likelihood ratio statistic as in \cite{AV14,BI13} to get
\[\tilde{L}=\binom{N}{n}^{-1}\sum_{|S|=n}L_SI_{\Gamma_S},\]
where $I_E$ is an indicator function for event $E$ and $\Gamma_S=\cap_{T\subset S,|T|\geq k_{min}+1} \{A_T\leq w_{|T|}=q_{|T|}|T|^{(m)}\}$. We will show $\mathbb{E}_0\tilde{L}=1+o(1)$ and $\mathbb{E}_0\tilde{L}^2\leq 1+o(1)$, where $\mathbb{E}_0$ represents expectation under $H_0$.

Consider the first-order moment. It is easy to verify $\mathbb{E}_0\tilde{L}=\mathbb{P}_S\Gamma_S$.
Then by Lemma \ref{lem:3}, it follows
\begin{eqnarray}\nonumber
\mathbb{P}_S\Gamma_S^c&=&\mathbb{P}_S(\cup_{T\subset S,|T|\geq k_{min}+1} \{A_T> w_{|T|}\})\\   \nonumber
&\leq&\sum_{k=k_{min}+1}^n\sum_{T\subset S,|T|=k}\mathbb{P}_S(A_T> w_{|T|})\\   \nonumber
&\leq&\sum_{k=k_{min}+1}^n\binom{n}{k}\mathbb{P}_S(Bin(k^{(m)},p_1)> q_kk^{(m)})\\   \nonumber
&\leq&\sum_{k=k_{min}+1}^ne^{k\log\frac{ne}{k}-k^{(m)}H_{p_1}(q_k)}\\   \nonumber
&=&\sum_{k=k_{min}+1}^ne^{k\left[\log\frac{ne}{k}-\frac{(k-1)(k-2)\cdots (k-m+1)}{m!}H_{p_1}(q_k)\right]}\rightarrow0.
\end{eqnarray}

Consider the second-order moment. Clearly, we have the following
\[\tilde{L}^2=\binom{N}{n}^{-2}\sum_{|S_1|=n,|S_2|=n}L_{S_1}I_{\Gamma_{S_1}}L_{S_1}I_{\Gamma_{S_1}}\]
and
\[
\mathbb{E}_0\tilde{L}^2=\mathbb{E}_0L_{S_1}L_{S_1}I_{\Gamma_{S_1}\cap\Gamma_{S_2}}=\mathbb{E}_0e^{\theta_{p_1}(A_{S_1}+A_{S_2})-2n^{(m)}\Lambda(\theta_{p_1})}I_{\Gamma_{S_1}\cap\Gamma_{S_2}}.
\]
Observe that $\Gamma_{S_1}\cap\Gamma_{S_2}\subset \{A_{S_1\cap S_2}\leq w_K\}$. Define $\tilde{S}_1=\{(i_1,\ldots,i_m)|\{i_1,\ldots,i_m\}\subset S_1, \{i_1,\ldots,i_m\}\not\subset S_1\cap S_2\}$ and $\tilde{S}_2=\{(i_1,\ldots,i_m)|\{i_1,\ldots,i_m\}\subset S_2, \{i_1,\ldots,i_m\}\not\subset S_1\cap S_2\}$. Then it's easy to verify that
$A_{S_1}+A_{S_2}=A_{\tilde{S}_1}+A_{\tilde{S}_2}+2A_{S_1\cap S_2}$,
and the tree terms are independent. By the definition of $\tilde{S}_1$ and $\tilde{S}_2$, one has $|\tilde{S}_1|=|\tilde{S}_2|=\binom{n}{m}-\binom{K}{m}$, where $K=|S_1\cap S_2|$.
Then we have
\[\mathbb{E}_0\tilde{L}^2\leq I\times II\times III,\]
where
\begin{eqnarray}\nonumber
I&=&\mathbb{E}_0e^{\theta_{p_1}A_{\tilde{S}_1}-\Lambda(\theta_{p_1})\left[\binom{n}{m}-\binom{k}{m}\right]}=1,\\ \nonumber
II&=&\mathbb{E}_0e^{\theta_{p_1}A_{\tilde{S}_2}-\Lambda(\theta_{p_1})\left[\binom{n}{m}-\binom{k}{m}\right]}=1,\\ \nonumber
III&=&\mathbb{E}_0e^{\theta_{p_1}A_{S_1\cap S_2}-\Lambda(\theta_{p_1})\left[2\binom{k}{m}\right]}I_{A_{S_1\cap S_2}\leq w_K}.
\end{eqnarray}
It is easy to check that
$III\leq e^{\Delta K^{(m)} }$ for $K\leq k_{min}$;
$III\leq e^{\Delta_K K^{(m)}}$ for $K>k_{min}$,
where $\Delta_K :=-2H_{p_1}(q_K)+H_{p_0}(q_K)$. Then
\[\mathbb{E}_0\tilde{L}^2\leq \mathbb{E}I[K\leq k_{min}]e^{\Delta K^{(m)} }+ \mathbb{E}I[K\geq k_{min}]e^{\Delta_K K^{(m)} }.\]

By condition (\ref{cond:2}), we can take a sequence $b$
with $b\rightarrow\infty$ such that
$\frac{p_1-p_0}{\sqrt{p_0}}\frac{b^{\frac{m}{2}}n^m}{N^{\frac{m}{2}}}=o(1)$.
Define $k_0=nb\rho\sim bn^2/N$ with $\rho=n/(N-n)$.

For $k_0\leq m-1$, we have
\[\mathbb{E}I[K\leq k_{min}]e^{\Delta K^{(m)} }\leq \mathbb{E}I[K\leq k_{min}]\leq 1.\]

For $k_0> m-1$, we have
\[\mathbb{E}I[K\leq k_{min}]e^{\Delta K^{(m)} }]\leq e^{\Delta k_0^{(m)} }\leq e^{\frac{(p_1-p_0)^2}{p_0(1-p_0)}(bn\rho)^m}\leq e^{\frac{O(1)(p_1-p_0)^2}{p_0}\frac{b^mn^{2m}}{N^m}}=1+o(1).\]

For $k_0+1\leq K\leq k_{min}$, we have
\begin{eqnarray}\nonumber
\mathbb{E}I[K\geq k_{min}]e^{\Delta K^{(m)} }&\leq& \sum_{k=k_0+1}^{k_{min}}P(K\geq k)e^{\Delta K^{(m)} }
=\sum_{k=k_0+1}^{k_{min}}e^{\Delta K^{(m)} -nH_{\rho}(\frac{k}{n})}\\
&=&\sum_{k=k_0+1}^{k_{min}}e^{k\left[\frac{(k-1)(k-2)\cdots (k-m+1)}{m!}\Delta  -\log\frac{k}{n\rho}+o(1)\right]},
\end{eqnarray}
where we have used the following fact
\begin{eqnarray*}nH_{\rho}(\frac{k}{n})&=&k\log\frac{k}{n\rho}+(n-k)\log(1-\frac{k}{n})-(n-k)\log(1-\rho)\\
&\sim& k\log\frac{k}{n\rho}+(n-k)\log(1-\frac{k}{n})-\frac{(n-k)n}{N}=k\log\frac{k}{n\rho}+o(k)
\end{eqnarray*}

Next we prove
\[\frac{(k-1)(k-2)\cdots (k-m+1)}{m!}\Delta  -\log\frac{k}{n\rho}\sim \frac{k^{m-1}}{m!}\Delta  -\log\frac{k}{n\rho}\rightarrow-\infty.\]
Note that the function $f(x)=\frac{x^{m-1}}{m!}\Delta -\log(x)+\log(n\rho)$ for $x>0$ attains minimum value at $x_0=\left(\frac{m!}{(m-1)\Delta}\right)^{\frac{1}{m-1}}< k_{min}$. Besides, it's increasing in $(x_0,\infty)$ and decreasing in $(0,x_0)$. Since $\frac{Nk_*}{n^2}\sim\frac{k_*}{n\rho}\sim \frac{k_{min}}{n\rho}$,
by definition of $k_{min}$, we have
$\frac{\Delta}{m!}k_{min}^{m-1}\leq \log\frac{k_{min}}{n\rho}-\log\log\frac{n^{m-1}}{\log\frac{N}{n}}$.
Hence, $f(k_{min})\rightarrow-\infty$. If $k_0> x_0$,  $f(k_0)\leq f(k)\leq f(k_{min})\rightarrow-\infty$.
If $k_0\leq x_0$, then $\frac{k_0^{m-1}\Delta}{m!}\leq 1$ and $\log\frac{k_0}{n\rho}=\log b\rightarrow\infty$. In this case, $f(k)\leq f(k_0)\vee f(k_{min})\rightarrow-\infty$.

For $K> k_{min}$, we have
\[\mathbb{E}I[K\geq k_{min}]e^{\Delta_K K^{(m)} }\leq \sum_{k=k_{min}}^{n}e^{k\left[\frac{(k-1)(k-2)\cdots (k-m+1)}{m!}\Delta_k  -\log\frac{k}{n\rho}+o(1)\right]}.\]
Note that
\begin{eqnarray}\nonumber
\frac{(k-1)(k-2)\cdots (k-m+1)}{m!}\Delta_k  -\log\frac{k}{n\rho}&=&-2\left(\frac{(k-1)(k-2)\cdots (k-m+1)}{m!}H_{p_1}(q_k)-\log\frac{n}{k}\right)\\ \nonumber
&+&\frac{(k-1)(k-2)\cdots (k-m+1)}{m!}H_{p_0}(q_k)-\log\frac{N}{k}+o(1),
\end{eqnarray}
which goes to $-\infty$ by Lemma \ref{lem2} and Lemma \ref{lem:3}. Then the proof is complete.
\end{proof}

\begin{proof}[Proof of Theorem \ref{thm:2} \ref{thm:2:I}]
 Under $H_0$, simple algebra yields
$\mathbb{E}_0W=N^{(m)}p_0$ and $\mathbb{V}_0W=N^{(m)}p_0(1-p_0)$.
Under $H_1$, we have
$\mathbb{E}_1W=N^{(m)}p_0+n^{(m)}(p_1-p_0)$ and
$\mathbb{V}_1W=N^{(m)}p_0(1-p_0)+n^{(m)}(p_1-p_0)(1-p_1-p_0)+(p_1-p_0)^2n^{2m-1}$.
Then
\[R:=\frac{\mathbb{E}_1W-\mathbb{E}_0W}{\sqrt{\mathbb{V}_1W+\mathbb{V}_0W}}=\frac{n^{(m)}(p_1-p_0)}{\sqrt{2N^{(m)}p_0(1-p_0)+n^{(m)}(p_1-p_0)(1-p_1-p_0)+(p_1-p_0)^2n^{2m-1}}}.\]
By condition (\ref{cond:4}), we have
$\frac{n^{(m)}(p_1-p_0)}{\sqrt{N^{(m)}p_0}}\rightarrow\infty$ and $\frac{n^{(m)}(p_1-p_0)}{\sqrt{(p_1-p_0)^2n^{2m-1}}}=\sqrt{n}\rightarrow\infty$.
If $n^{(m)}(p_1-p_0)(1-p_1-p_0)>N^{(m)}p_0$, then by condition (\ref{cond:4}),
we have
\[\frac{n^{(m)}(p_1-p_0)}{\sqrt{n^{(m)}(p_1-p_0)(1-p_1-p_0)}}>\sqrt{n^{(m)}(p_1-p_0)}\gg (N^{(m)}p_0)^{\frac{1}{4}}\rightarrow\infty.\]
Hence, $R\rightarrow\infty$. The proof is completed by Lemma 9 in \cite{AV14}.
\end{proof}

\begin{proof}[Proof of Theorem \ref{thm:2} \ref{thm:2:II}] Let $a=\eta p_0+(1-\eta) p_1$ for some $\eta$ with  $\eta=o(1)$ such that
\[
\limsup\limits_{N,n\to\infty} \frac{(n-1)(n-2)\cdots(n-m+1)H_{p_0}(a)}{m!\log\frac{N}{n}}>1,
\]
which is possible by condition (\ref{cond:5}). Then under $H_0$, one has
\begin{eqnarray*}
\mathbb{P}_0(W_n\geq a n^{(m)})&\leq& \binom{N}{n}\mathbb{P}_0(W_S\geq a n^{(m)})\\
&\leq& \binom{N}{n}e^{-n^{(m)}H_{p_0}(a)}\leq e^{n\left[\log\frac{Ne}{n}-\frac{(n-1)(n-2)\cdots(n-m+1)}{m!}H_{p_0}(a)\right]}=o(1).
\end{eqnarray*}

Under $H_1$, for a fixed $S_1$, it's easy to get $\mathbb{E}_1W_{S_1}=n^{(m)}p_1$ and  $\mathbb{V}_1W_{S_1}=n^{(m)}p_1(1-p_1)$. Then 
$W_{S_1}=n^{(m)}p_1+O_p(\sqrt{n^{(m)}p_1}))$,
from which it follows
\[\mathbb{P}_1(W_n\geq a n^{(m)})=\mathbb{P}_1\left(\frac{W_n-W_{S_1}}{\sqrt{n^{(m)}p_1}}\geq \frac{a n^{(m)}-W_{S_1}}{\sqrt{n^{(m)}p_1}}\right)=\mathbb{P}_1\left(\frac{W_n-W_{S_1}}{\sqrt{n^{(m)}p_1}}\geq \frac{n^{(m)}(a-p_1)}{\sqrt{n^{(m)}p_1}}+O_p(1)\right).\]
Note that
$\frac{n^{(m)}(p_1-a)}{\sqrt{n^{(m)}p_1}}=\frac{\eta n^{(m)}(p_1-p_0)}{\sqrt{n^{(m)}p_1}}$.
If $\frac{p_1}{p_0}\rightarrow r\in (1,\infty]$, let $\eta=(n^{(m)}p_1)^{-\frac{1}{4}}=o(1)$. Then
\[\frac{n^{(m)}(p_1-a)}{\sqrt{n^{(m)}p_1}}\sim \frac{\eta n^{(m)}p_1(1-\frac{1}{r})}{\sqrt{n^{(m)}p_1}}\rightarrow\infty.\]
If $\frac{p_1}{p_0}\rightarrow 1$, let $\eta=(\log\frac{N}{n})^{-\frac{1}{4}}=o(1)$. In this case, by $H_{p_0}(p_1)>\frac{m!\log\frac{N}{n}}{n^{m-1}}$, one has
\[\left(\frac{\eta n^{(m)}(p_1-p_0)}{\sqrt{n^{(m)}p_1}}\right)^2\sim\frac{\eta^2 n^{(m)}(p_1-p_0)^2}{p_0}>n\eta^2\log\frac{N}{n}\rightarrow\infty.\]
Proof is complete.
\end{proof}

\begin{Lemma}\label{Lemma2} Under the condition of Theorem \ref{thm:4},
\begin{equation}\label{cond4:4}
 \frac{n^{m}}{N}\frac{(p_1-p_0)^2}{p_0(1-p_0)}=o(1),\,\, \frac{n^{2m-2}}{N^{m-1}}\frac{(p_1-p_0)^2}{p_0(1-p_0)}=o(1),\,\, \frac{n^{2m-1}}{N^m}\frac{(p_1-p_0)^2}{p_0(1-p_0)}=o(1).
\end{equation}
\end{Lemma}

\begin{proof}[Proof of Lemma \ref{Lemma2}] The proof is almost the same as
 Lemma 2 in \cite{AV14}, so is omitted.
\end{proof}

\begin{Lemma}\label{Lemma1} Under the condition of Theorem \ref{thm:4}, it holds that
\begin{equation}\label{cond4:3}
\frac{p_1-p_0}{\sqrt{p_0}}\left(\frac{n^2}{N}\right)^{\frac{m+1}{4}}=o(1),\,\,\,\,
\limsup\limits_{N,n\to\infty} \frac{(n-1)(n-2)\cdots(n-m+1)H_{p_0}(p_1)}{m!\log\frac{N}{n}}<1.
\end{equation}
\end{Lemma}

\begin{proof}[Proof of Lemma \ref{Lemma1}] 
The first conclusion in (\ref{cond4:3}) directly follows from (\ref{cond4:4}).
By Lemma \ref{Lemma2}, it holds that
\[H_{p_0^{\prime}}(p_1)-H_{p_0}(p_1)\leq \frac{n^{(m)}}{N^{(m)}}\frac{(p_1-p_0)^2}{p_0^{\prime}(1-p_0^{\prime})}=o(1).\]
Then the second conclusion in (\ref{cond4:3}) follows. Proof is completed. 
\end{proof}

\begin{proof}[Proof of Theorem \ref{thm:4}] Conditional on a given subset $S\subset[N]$, the likelihood ratio statistic is expressed as
\begin{eqnarray}\nonumber
L_S&=&\frac{\prod_{\substack{i_1<\cdots<i_m\\ \{i_1,\ldots, i_m\}\subset S}}p_1^{A_{i_1\ldots i_m}}(1-p_1)^{1-A_{i_1\ldots i_m}}\prod_{\substack{i_1<\cdots<i_m\\ \{i_1,\ldots, i_m\}\not\subset S}}(p_0^{\prime})^{ A_{i_1\ldots i_m}}(1-p_0^{\prime})^{1-A_{i_1\ldots i_m}}}{\prod_{\substack{i_1<\cdots<i_m\\ \{i_1,\ldots, i_m\}\subset S}}p_0^{A_{i_1\ldots i_m}}(1-p_0)^{1-A_{i_1\ldots i_m}}\prod_{\substack{i_1<\cdots<i_m\\ \{i_1,\ldots, i_m\}\not\subset S}}p_0^{A_{i_1\ldots i_m}}(1-p_0)^{1-A_{i_1\ldots i_m}}}\\ \nonumber
&=&\left(\frac{p_1(1-p_0)}{p_0(1-p_1)}\right)^{A_S}\left(\frac{1-p_1}{1-p_0}\right)^{n^{(m)}}  \left(\frac{p_0^{\prime}(1-p_0)}{p_0(1-p_0^{\prime})}\right)^{A-A_S}\left(\frac{1-p_0^{\prime}}{1-p_0}\right)^{N^{(m)}-n^{(m)}}  \\
&=&e^{\theta_{p_1}A_S-n^{(m)}\Lambda(\theta_{p_1})+\theta_{p_0^{\prime}}(A-A_S)-\theta_{p_0^{\prime}}(N^{(m)}-n^{(m)}) }.
\end{eqnarray}
Consider the truncated uncondiitonal likelihood statistic
\[\tilde{L}=\binom{N}{n}^{-1}\sum_{|S|=n}L_SI_{\Gamma_S},\]
where $\Gamma_S$ is defined in the proof of Theorem \ref{thm:1}. It's easy to see $\mathbb{E}_0\tilde{L}=1+o(1)$ as in the proof of Theorem \ref{thm:1}. Hence we only need to prove $\mathbb{E}_0\tilde{L}^2\leq 1+o(1)$. 
For two subsets $S_1$ and $S_2$, let $K=|S_1\cap S_2|$. Note that $A_{S_1}+A_{S_2}-2A_{S_1\cap S_2}$, $A_{S_1\cap S_2}$, $A-A_{S_1}-A_{S_2}+A_{S_1\cap S_2}$ are independent. Then we have
\[\tilde{L}^2=\binom{N}{n}^{-2}\sum_{|S_1|=n,|S_2|=n}L_{S_1}I_{\Gamma_{S_1}}L_{S_1}I_{\Gamma_{S_1}},\]
and
\begin{eqnarray}\nonumber
\mathbb{E}_0\tilde{L}^2&=&\mathbb{E}_0L_{S_1}L_{S_1}I_{\Gamma_{S_1}\cap\Gamma_{S_2}}=\mathbb{E}_0e^{\theta_{p_1}(A_{S_1}+A_{S_2})-2n^{(m)}\Lambda(\theta_{p_1})+\theta_{p_0^{\prime}}(2A-A_{S_1}-A_{S_2})-2\theta_{p_0^{\prime}}(N^{(m)}-n^{(m)})}I_{\Gamma_{S_1}\cap\Gamma_{S_2}}\\
&\leq &I\times II\times III,
\end{eqnarray}
where
\begin{eqnarray}\nonumber
I&=&\mathbb{E}_0e^{2 \theta_{p_0^{\prime}}(A-A_{S_1}-A_{S_2}+A_{S_1\cap S_2})-2\Lambda_{p_0^{\prime}}(N^{(m)}-2n^{(m)}+K^{(m)})},\\   \nonumber
II&=&\mathbb{E}_0e^{(\theta_{p_1}+ \theta_{p_0^{\prime}})(A_{S_1}+A_{S_2}-2A_{S_1\cap S_2})-2(\Lambda_{p_1}+\Lambda_{p_0^{\prime}})(n^{(m)}-K^{(m)})},\\   \nonumber
III&=&\mathbb{E}_0e^{2\theta_{p_1}A_{S_1\cap S_2}-2\Lambda_{p_1}K^{(m)}}I[A_{S_1\cap S_2}\leq w_K].
\end{eqnarray}
For $III$, we have
$III\leq e^{\Delta K^{(m)}}$ if $K\leq k_{min}$;
and $III\leq e^{\Delta_K K^{(m)}}$ if $K> k_{min}$.
For $I$ and $II$, we get the following upper bounds:
\[I\leq e^{(N^{(m)}-2n^{(m)}+K^{(m)})\frac{(p_1-p_0^{\prime})^2}{p_0(1-p_0)}\frac{n^{(m)2}}{N^{(m)2}}},\,\,\,\,II\leq e^{-2(n^{(m)}-K^{(m)})\frac{(p_1-p_0)(p_1-p_0^{\prime})}{p_0(1-p_0)}\frac{n^{(m)}}{N^{(m)}}}.\]
Simple algebra yields $p_1-p_0=(p_1-p_0^{\prime})(1-\frac{n^{(m)}}{N^{(m)}})$. Then one gets
\begin{eqnarray}\nonumber
I\times II&\leq& e^{\frac{(p_1-p_0)^2}{p_0(1-p_0)}\left[-\frac{n^{(m)2}}{N^{(m)}}\left(1-\frac{n^{(m)}}{N^{(m)}}\right)^{-2}+2\frac{K^{(m)}n^{(m)}}{N^{(m)}}\left(1-\frac{n^{(m)}}{N^{(m)}}\right)^{-2}- \frac{K^{(m)}n^{(m)2}}{N^{(m)2}}\left(1-\frac{n^{(m)}}{N^{(m)}}\right)^{-2}\right]}\\
&=&e^{\frac{(p_1-p_0)^2}{p_0(1-p_0)}\left[-\frac{n^{(m)2}}{N^{(m)}}+\frac{n^{(m)}}{N^{(m)}}\left(K^{(m)}-\frac{n^{(m)2}}{N^{(m)}}\right)\left(1-\frac{n^{(m)}}{N^{(m)}}\right)^{-2}\left(2-\frac{n^{(m)}}{N^{(m)}}\right)\right]}=V_K.
\end{eqnarray}
By Lemma \ref{Lemma2}, it's easy to get
\[V_K\leq e^{\frac{(p_1-p_0)^2}{p_0(1-p_0)}\frac{n^{(m)}}{N^{(m)}}K^{(m)}}\leq e^{\frac{(p_1-p_0)^2}{p_0(1-p_0)}\frac{n^{2m-1}}{N^m}K}=e^{o(K)}.\]

Take a sequence $b\to\infty$ such that 
\[\frac{(p_1-p_0)^2}{p_0}\left(\frac{n^2}{N}\right)^{\frac{m}{2}}b^m=o(1),\,\,\,\, \frac{(p_1-p_0)^2}{p_0}\frac{n^{2m-2}}{N^{m-1}}b^{m-1}=o(1).\]
Set $k_0=b\frac{n^2}{N}\leq \frac{n}{2}$ and $k_0^{\prime}=\frac{n^2}{N}+\frac{n}{\sqrt{N}}b$.  We only need to show that
\begin{eqnarray}\label{eq4}
\mathbb{E}[I[K\leq k_0^{\prime}]e^{\Delta K^{(m)}}V_K]&\leq& 1+o(1),\\  \label{eq1}
\mathbb{E}[I[k_0^{\prime} < K\leq k_0]e^{\Delta K^{(m)}}V_K]&=&o(1),\\  \label{eq2}
\mathbb{E}[I[k_0< K\leq k_{min}]e^{\Delta K^{(m)}}V_K]&=&o(1),\\  \label{eq3}
\mathbb{E}[I[k_{min}< K\leq n ]e^{\Delta_K K^{(m)}}V_K]&=&o(1).   
\end{eqnarray}
Since $V_K=e^{o(K)}$, (\ref{eq2}) and (\ref{eq3}) are true by the proof of Theorem \ref{thm:1}. We next prove (\ref{eq4}) and (\ref{eq1}). Note that $K\leq k_0\leq \frac{n}{2}$. Therefore,
\[e^{\Delta K^{(m)}}V_K\leq e^{\Delta(1+o(1)) \left(K^{(m)}-\frac{n^{(m)2}}{N^{(m)}}\right)} \leq e^{(1+o(1))\frac{\Delta}{m!} \left(K^{m}-\frac{n^{m2}}{N^m}\right)+O\left(\frac{\Delta n^{2m-1}}{N^m}\right)}=(1+o(1))e^{(1+o(1))\frac{\Delta}{m!} \left(K^{m}-\frac{n^{2m}}{N^m}\right)}.  \]
For $k_0^{\prime}\leq K\leq k_0$, by simple algebra and the definition of $k_0$, we obtain
\[
K^{m}-\frac{n^{2m}}{N^m}=\left(K-\frac{n^{2}}{N}\right)\left(K^{m-1}+\cdots+\frac{n^{2(m-1)}}{N^{m-1}}\right)\leq \left(K-\frac{n^{2}}{N}\right)mb^{m-1}\left(\frac{n^2}{N}\right)^{m-1}.
\]
Hence,
\begin{eqnarray}\nonumber
\mathbb{E}e^{\Delta K^{(m)}}V_K&\leq& \mathbb{E}e^{\frac{\Delta}{m!}mb^{m-1}\left(\frac{n^2}{N}\right)^{m-1}(K-\frac{n^{2}}{N})}   \leq \left[1-\frac{n}{N}+\frac{n}{N}e^{\frac{\Delta}{m!}mb^{m-1}\left(\frac{n^2}{N}\right)^{m-1}}\right]^ne^{-\frac{\Delta}{m!}mb^{m-1}\left(\frac{n^2}{N}\right)^{m}}\\ \nonumber
&=&e^{\frac{\Delta}{m!}mb^{m-1}\left(\frac{n^2}{N}\right)^{m}+O\left(\Delta^2\left(\frac{n^2}{N}\right)^{2(m-1)}b^{2(m-1)}\right) }e^{-\frac{\Delta}{m!}mb^{m-1}\left(\frac{n^2}{N}\right)^{m}}=1+o(1),
\end{eqnarray}
where we have used the fact that $\Delta b^{m-1}\left(\frac{n^2}{N}\right)^{m-1}=o(1)$. By Chebyshev's inequality, we have $\mathbb{P}(K>k_0^{\prime})\leq b^{-2}=o(1)$. Hence, (\ref{eq1}) holds.

When $K\leq k_0^{\prime}$, by Lemma \ref{Lemma1}, (\ref{eq1}) follows by
\begin{eqnarray}\nonumber
\mathbb{E}[I[K&\leq& k_0^{\prime}]e^{\Delta K^{(m)}}V_K]\leq (1+o(1))e^{(1+o(1))\frac{\Delta}{m!} \left(K^{m}-\frac{n^{2m}}{N^m}\right)}\\  \nonumber
&=&(1+o(1))e^{(1+o(1))\frac{\Delta}{m!} \left(k_0^{\prime}-\frac{n^{2}}{N}\right)\left(k_0^{\prime (m-1)}+\cdots+\frac{n^{2(m-1)}}{N^{m-1}}\right)}\\  \nonumber
&=&(1+o(1))e^{\frac{\Delta}{m!} \frac{bn}{\sqrt{N}}\left(k_0^{\prime m-1}+\cdots+\frac{n^{2(m-1)}}{N^{m-1}}\right)}=1+o(1).
\end{eqnarray}
This completes the proof.
\end{proof}

\begin{proof}[Proof of Theorem \ref{thm:6} \ref{thm:6:I}]
For type I error, let $q=p_0-\sqrt{\frac{p_0\log\frac{N}{n}}{N^{(m)}}}$. 
By Markov inequality, it's easy to verify that $\mathbb{P}(\hat{p}_0\leq q)=o(1)$. Let $\hat{a}_q=H^{-1}_{\hat{p}_0}(n\frac{\log\frac{N}{n}+2}{n^{(m)}})$ and $a_q=H^{-1}_{q}(n\frac{\log\frac{N}{n}+2}{n^{(m)}})$. Note that
\begin{eqnarray}\label{unkonwntypeI1}
\mathbb{P}\left(W^*_n\geq n^{(m)}\hat{a}_q\right)&=&\mathbb{P}\left(W^*_n\geq n^{(m)}\hat{a}_q\cap \hat{p}_0>q \right)+\mathbb{P}\left(W^*_n\geq n^{(m)}\hat{a}_q\cap \hat{p}_0\leq q \right)\\  \nonumber
&\leq&\mathbb{P}\left(W^*_n\geq n^{(m)}a_q\right)+o(1)\\  \nonumber
&\leq& \sum_{|S|=n}\mathbb{P}\left(W_S\geq n^{(m)}a_q\right)+o(1)\\   \nonumber
&\leq&\binom{N}{n}e^{-n^{(m)}H_{p_0}(a_q)}+o(1)\\   \nonumber
&\leq& e^{-n\left(1+\frac{n^{(m)}}{n}(H_q(a_q)-H_{p_0}(a_q))\right)}+o(1),
\end{eqnarray}
where the last step follows from the fact that $n\log\frac{Ne}{n}=-n+n^{(m)}H_q(a_q)$ and $\log\binom{N}{n}\leq n\log\frac{Ne}{n}$. We will show $n^{m-1}(H_q(a_q)-H_{p_0}(a_q))=o(1)$ 
in two cases which together conclude that $\mathbb{P}\left(W^*_n\geq n^{(m)}\hat{a}_q\right)=o(1)$ under $H_0$.

\textit{Case (1).} Suppose $n^{m-1}p_0\gg \log\frac{N}{n}$. Then $N^{m-1}p_0\gg \log\frac{N}{n}$. By straightforward calculations, we have
\begin{eqnarray}\label{unknowntypeI2}
H_q(a_q)-H_{p_0}(a_q)&=&a_q\log\frac{p_0}{q}+(1-a_q)\log\frac{1-p_0}{1-q}\\  \nonumber
&=&a_q\log\left(1+\frac{\sqrt{\frac{p_0\log\frac{N}{n}}{N^{(m)}}}}{q}\right)+(1-a_q)\log\left(1-\frac{\sqrt{\frac{p_0\log\frac{N}{n}}{N^{(m)}}}}{1-q}\right)\\  \nonumber
&=&\frac{a_q}{q}\sqrt{\frac{p_0\log\frac{N}{n}}{N^{(m)}}}-(1-a_q)\frac{\sqrt{\frac{p_0\log\frac{N}{n}}{N^{(m)}}}}{1-q}+O\left(\frac{a_q}{q}\frac{\log\frac{N}{n}}{N^m}\right)\\   \nonumber
&=&\frac{1}{1-q}\left(\frac{a_q}{q}-1\right)\sqrt{\frac{p_0\log\frac{N}{n}}{N^{(m)}}}+O\left(\frac{a_q}{q}\frac{\log\frac{N}{n}}{N^m}\right).
\end{eqnarray}
We need to show that $1\leq\frac{a_q}{q}\leq 1+O\left(\sqrt{\frac{\log\frac{N}{n}}{n^{m-1}p_0}}\right)$. It's clear that $q\leq a_q$ since $H_q(q)=0<H_q(a_q)$. Let $\delta=1+\sqrt{n\frac{\log\frac{N}{n}}{N^{(m)}p_0}}$. Then
\begin{eqnarray}\nonumber
H_q(\delta p_0)&=&\delta p_0\log \delta+\delta p_0\log\frac{p_0}{q}+(1-\delta p_0)\log\left(1+\frac{q-\delta p_0}{1-q}\right)\\  \nonumber
&\sim&\delta p_0\left[\sqrt{n\frac{\log\frac{N}{n}}{n^{(m)}p_0}}+\sqrt{\frac{\log\frac{N}{n}}{N^{(m)}p_0}}-\frac{1-\delta p_0}{\delta(1-q)}\left(\sqrt{\frac{n\log\frac{N}{n}}{n^{(m)}p_0}}+\sqrt{\frac{\log\frac{N}{n}}{N^{(m)}p_0}}\right)\right]\\  \nonumber
&\sim&\delta p_0\left(1-\frac{1-\delta p_0}{\delta(1-q)}\right)\sqrt{\frac{n\log\frac{N}{n}}{n^{(m)}p_0}}\\  \nonumber
&=&\delta p_0\frac{\sqrt{\frac{n\log\frac{N}{n}}{n^{(m)}p_0}}+\sqrt{\frac{\log\frac{N}{n}p_0}{N^{(m)}}}}{\delta (1-q)}\sqrt{\frac{n\log\frac{N}{n}}{n^{(m)}p_0}}\\  \nonumber
&=&\frac{1}{1-q}\frac{n\log\frac{N}{n}}{n^{(m)}}(1+o(1))>H_q(a_q).
\end{eqnarray}
The second equation  in the above follows from the fact that $\frac{q}{\delta p_0}-1=-\frac{1}{\delta}\left(\sqrt{\frac{n\log\frac{N}{n}}{n^{(m)}p_0}}+\sqrt{\frac{\log\frac{N}{n}}{N^{(m)}p_0}}\right)$. Hence, it holds that  $a_q\leq \delta p_0\sim \delta q$ and $1\leq\frac{a_q}{q}\leq 1+O\left(\sqrt{\frac{\log\frac{N}{n}}{n^{m-1}p_0}}\right)$.

\textit{Case (2).} Suppose $n^{m-1}p_0=O(\log\frac{N}{n})$. Note that $q=p_0-\sqrt{\frac{p_0\log\frac{N}{n}}{N^{(m)}}}$. By straightforward calculations, we have
\begin{eqnarray}\label{unknowntypeI3}
H_q(a_q)-H_{p_0}(a_q)&=&a_q\log\frac{a_q(1-q)}{q(1-a_q)}-a_q\log\frac{a_q(1-p_0)}{p_0(1-a_q)}+o(1)\\ \nonumber
&=&a_q\log\frac{p_0(1-q)}{q(1-p_0)}+o(1)\\  \nonumber
&\leq&a_q\frac{p_0-q}{q(1-p_0)}+o(1)\leq a_q\sqrt{\frac{\log\frac{N}{n}}{N^mp_0}}+o(1).
\end{eqnarray}
We need to show that $r_n=n^{m-1} a_q\sqrt{\frac{\log\frac{N}{n}}{N^mp_0}}=o(1)$. If $n^{m-1}p_0\asymp\log\frac{N}{n}$, then $a_q\asymp p_0$ and $r_n=o(1)$. 
Suppose $n^{m-1}p_0=o(\log\frac{N}{n})$. If $N^{m-1}p_0\asymp 1$, then $H_q(\frac{\log\frac{N}{n}}{n^{m-1}})\gg H_q(a_q)$, which yields $a_q<\frac{\log\frac{N}{n}}{n^{m-1}}$. Hence $r_n=(1)$. If $N^{m-1}p_0=o(1)$, then $p_0=o(\frac{1}{n^{m-1}})$. Consequently, $H_q(\frac{1}{n^{m-1}})\gg H_q(a_q)$, which implies $a_q\leq \frac{1}{n^{m-1}}$ and hence $r_n=o(1)$.

For type II error, we need to prove
\begin{equation}\label{cond6:1}
\lim\sup \frac{(n-1)(n-2)\cdots(n-m+1)H_{\hat{p}_0}(p_1)}{m!\log\frac{N}{n}}>1.
\end{equation}
Direct calculations yield that
\[\mathbb{E}m!N^{(m)}\hat{p}_0=m!(N^{(m)}p_0^{\prime}+n^{(m)}(p_1-p_0^{\prime})) =m!N^{(m)}p_0, \]
\[\mathbb{V}m!N^{(m)}\hat{p}_0 \leq m!N^{(m)}p_0. \]
As a result, we have $\hat{p}_0=p_0+O_p\left(\sqrt{\frac{p_0}{N^m}}\right)$.
By the condition $N^mp_0>2n$, it follows $\hat{p}_0=p_0(1+o_p(1))$. Let $a=\frac{n^{(m)}}{N^{(m)}}(p_1-p_0^{\prime})$.

If $\frac{p_1}{p_0^{\prime}}\rightarrow1$, then $a=o(p_0^{\prime})$ and
$\frac{(p_1-p_0^{\prime})^2}{p_0^{\prime}}\succ H_{p_0^{\prime}}(p_1)\succ \frac{m!\log\frac{N}{n}}{n^{m-1}}$,
which implies $\sqrt{\frac{p_0^{\prime}}{N^{(m)}}}=o(p_1-p_0^{\prime})$.
As a result, $\hat{p}_0-p_0^{\prime}=o(p_1-p_0^{\prime})$, and then $H_{\hat{p}_0}(p_1)\geq H_{p_0}(p_1)$ asymptotically.

If $\frac{p_1}{p_0^{\prime}}\rightarrow r\in(1,\infty)$, then $a=o(p_0^{\prime})$, it's easy to obtain $H_{\hat{p}_0}(p_1)\sim H_{p_0}(p_1)$.

If $\frac{p_1}{p_0^{\prime}}\rightarrow\infty$, we have 
$H_{\hat{p}_0}(p_1)\geq H_{p_0^{\prime}}(p_1)\wedge mp_1\log\frac{N}{n}$.
\end{proof}

\begin{proof}[Proof of Theorem \ref{thm:6} \ref{thm:6:II}]  
Simple algebra yields that
\[\sum_{i_1=1}^N\left[W_{i_1*}-(m-1)!(N-1)^{(m-1)}p_0\right]=m!N^{(m)}(\hat{p}_0-p_0).\]
Meanwhile, it's easy to verify that
\begin{eqnarray}\nonumber
&&\sum_{i_1=1}^N\left[W_{i_1*}-(m-1)!(N-1)^{(m-1)}\hat{p}_0\right]^2\\
&=&\sum_{i_1=1}^N\left[W_{i_1*}-(m-1)!(N-1)^{(m-1)}p_0\right]^2-(m-1)!^2N(N-1)^{(m-1)2}(\hat{p}_0-p_0)^2,
\end{eqnarray}
and
\[\hat{p}_0(1-\hat{p}_0)=\hat{p}_0(1-p_0)-(\hat{p}_0-p_0)^2+p_0(\hat{p}_0-p_0).\]

Next, we show $\mathcal{T}_1=O_p(1)$ under $H_0$. Note that
\begin{equation}\label{T1eq1}
\mathbb{E}(\hat{p}_0-p_0)^2=\frac{p_0(1-p_0)}{N^{(m)}},
\end{equation}
\begin{equation}\label{T1eq2}
\mathbb{V}[(\hat{p}_0-p_0)^2]\leq \frac{p_0^2(1-p_0)^2}{N^{(m)2}}+\frac{p_0(1-p_0)}{N^{(m)3}},
\end{equation}
\[\mathbb{E}\left[W_{i_1*}-(m-1)!(N-1)^{(m-1)}p_0\right]^2=(m-1)!(N-1)^{(m-1)}p_0(1-p_0),\]
\[\mathbb{V}\left(\sum_{i_1=1}^N\left[W_{i_1*}-(m-1)!(N-1)^{(m-1)}p_0\right]^2\right)\leq  N^{2m-1}p_0^2+N^{m-2}p_0+N^mp_0.\]
Since
$\mathbb{E}V_2=\mathbb{E}V_1=(m-1)!(N-1)^{(m-1)} p_0(1-p_0)$,
we get $\mathbb{E}V=0$.
Besides, since $N^{m-1}p_0>1$, then
\[\mathbb{V}(V)\leq \frac{N^{2m-1}p_0^2+N^{m-2}p_0+N^mp_0}{N^2}+N^{m-2}(p_0^2+p_0)\leq N^{2m-3}p_0^2.\]
By Chebyshev's inequality, we get $V=O_p(N^{\frac{2m-3}{2}}p_0)$. Note that
\[\mathbb{P}(\hat{p}_0<0.5p_0)\leq \mathbb{P}(|\hat{p}_0-p_0|\geq 0.5p_0)\leq 4\frac{1-p_0}{N^{(m)}p_0}=o(1). \]
Hence, $V=O_p(N^{\frac{2m-3}{2}}\hat{p}_0)$ and $\mathcal{T}_1=O_p(1)$ under $H_0$.

In the following, let us show that $\mathbb{E}(V)\gg N^{\frac{2m-3}{2}}p_0 $ and $\mathbb{E}(V)\gg \sqrt{\mathbb{V}(V)} $ under $H_1$. Note that for a given subset of vertices $S$, 
\[\hat{p}_0-p_0=\frac{1}{N^{(m)}}\left[\sum_{\substack{ i_1<i_2<\cdots<i_m,\\ \{i_1,\ldots,i_m\}\subset S}}(A_{i_1\ldots i_m}-p_1)+\sum_{\substack{ i_1<i_2<\cdots<i_m,\\ \{i_1,\ldots,i_m\}\not\subset S}}(A_{i_1\ldots i_m}-p_0^{\prime})\right].\]
Then 
\begin{eqnarray}\label{unk1}
\mathbb{E}(\hat{p}_0-p_0)^2&=&\frac{1}{N^{(m)2}}\left(n^{(m)}p_1(1-p_1)+(N^{(m)}-n^{(m)})p_0^{\prime}(1-p_0^{\prime})\right)\\  \nonumber
&=&\frac{1}{N^{(m)2}}\left\{N^{(m)}p_0-N^{(m)}p_0^2+N^{(m)}p_0^2-n^{(m)}p_1^2-(N^{(m)}-n^{(m)})p_0^{\prime 2}\right\} \\  \nonumber
&=& \frac{p_0(1-p_0)}{N^{(m)}}-\frac{n^{(m)}(p_1-p_0^{\prime})^2}{N^{(m)2}}\left(1+o(1)\right),
\end{eqnarray}
and
\begin{eqnarray}\nonumber
&&\mathbb{E}\left(\sum_{i_1=1}^N\left[W_{i_1*}-(m-1)!(N-1)^{(m-1)}p_0\right]^2\right)\\
&=&\sum_{\substack{i_1\neq i_2\neq\cdots\neq i_m\\ i_1\neq j_2\neq\cdots\neq j_m}}\mathbb{E}\left[A_{i_1i_2\ldots i_m}A_{i_1j_2\ldots j_m}\right] - m!(m-1)!N^{(m)}(N-1)^{(m-1)} p_0^2.
\end{eqnarray}
Before calculating the above expectation,
with a little abuse of notation, we define 
\[
S_{i_1}=\{(i_2,\ldots,i_m)| 
\text{$i_1,\ldots,i_m$ are pairwise different and} \{i_2,\ldots,i_m\}\subset S\},
\]
\[
\bar{S}_{i_1}=\{(i_2,\ldots,i_m)| \text{$i_1,\ldots,i_m$ are pairwise different and} \{i_2,\ldots,i_m\}\not\subset S\},
\]
\[
U_{i_1}=\{(i_2,\ldots,i_m)| \text{$i_1,\ldots,i_m$ are pairwise different}\}.
\]
Then
\begin{eqnarray*}\nonumber
&&\sum_{\substack{i_1\neq i_2\neq\cdots\neq i_m\\ i_1\neq j_2\neq\cdots\neq j_m}}\mathbb{E}\left[A_{i_1i_2\ldots i_m}A_{i_1j_2\ldots j_m}\right]\\   \nonumber
 &=&\sum_{i_1\in S,S_{i_1},S_{i_1}}\mathbb{E}\left[A_{i_1i_2\ldots i_m}A_{i_1j_2\ldots j_m}\right] +\sum_{i_1\in S,\bar{S}_{i_1},S_{i_1}}\mathbb{E}\left[A_{i_1i_2\ldots i_m}A_{i_1j_2\ldots j_m}\right] +\sum_{i_1\in S,S_{i_1},\bar{S}_{i_1}}\mathbb{E}\left[A_{i_1i_2\ldots i_m}A_{i_1j_2\ldots j_m}\right] \\  \nonumber
&&+\sum_{i_1\in S,\bar{S}_{i_1},\bar{S}_{i_1}}\mathbb{E}\left[A_{i_1i_2\ldots i_m}A_{i_1j_2\ldots j_m}\right] +\sum_{i_1\notin S,U_{i_1},U_{i_1}}\mathbb{E}\left[A_{i_1i_2\ldots i_m}A_{i_1j_2\ldots j_m}\right]\\ \nonumber
&=&\left\{(m-1)!^2n(n-1)^{(m-1)2}p_1^2-(m-1)!n(n-1)^{(m-1)}p_1^2\right.\\   \nonumber
&&+2(m-1)!^2n(n-1)^{(m-1)}\left[ (N-1)^{(m-1)}-(n-1)^{(m-1)}\right]p_1p_0^{\prime} \\   \nonumber
&&+(m-1)!^2n\left[ (N-1)^{(m-1)}-(n-1)^{(m-1)}\right]^2p_0^{\prime 2}-(m-1)!n\left[ (N-1)^{(m-1)}-(n-1)^{(m-1)}\right]p_0^{\prime 2}\\   \nonumber
&&\left.+ (m-1)!^2(N-n)(N-1)^{(m-1)2}p_0^{\prime 2}-(m-1)!(N-n)(N-1)^{(m-1)}p_0^{\prime 2}\right\}\left(1+o(1)\right),
\end{eqnarray*}
where $\sum_{i_1\in S, \bar{S}_{i_1},S_{i_1}}$ is a summation over all $i_1\in S$, $(i_2,\ldots, i_m)\in \bar{S}_{i_1}$ and $(j_2,\ldots, j_m)\in S_{i_1}$, the rest summations are similarly defined. As a result, we get
\begin{eqnarray}\nonumber
&&\mathbb{E}\left(\sum_{i_1=1}^N\left[W_{i_1*}-(m-1)!(N-1)^{(m-1)}p_0\right]^2\right)\\
&=&m!N^{(m)}p_0(1-p_0)+(m-1)!^2n(n-1)^{(m-1)2}(p_1-p_0^{\prime})^2\left(1+o(1)\right).
\end{eqnarray}
Consequently, we obtain
\[\mathbb{E}V\sim\frac{n^{2m-1}(p_1-p_0^{\prime})^2 }{N-m!}-\frac{n^m(p_1-p_0^{\prime})^2 }{N^2}+\frac{n^m(p_1-p_0^{\prime})^2 }{N^{m+1}}\sim \frac{n^{2m-1}(p_1-p_0^{\prime})^2 }{N}. \]

If $p_0^{\prime}\asymp (p_1-p_0^{\prime})\frac{n^{(m)}}{N^{(m)}}$, then $p_0\asymp p_0^{\prime}$. By condition (\ref{cond7:1}), we have
\[\frac{n^{2m-1}(p_1-p_0^{\prime})^2 }{NN^{\frac{2m-3}{2}}p_0^{\prime}}=\frac{(p_1-p_0^{\prime})^2}{p_0^{\prime}}\left(\frac{n^2}{N}\right)^{\frac{2m-1}{2}}\rightarrow\infty.\]
Hence, $\mathbb{E}V\gg N^{\frac{2m-3}{2}}p_0$. If $p_0^{\prime}\ll (p_1-p_0^{\prime})\frac{n^{(m)}}{N^{(m)}}$, then $p_0\asymp (p_1-p_0^{\prime})\frac{n^{(m)}}{N^{(m)}}$. In this case, $\mathbb{E}V\gg N^{\frac{2m-3}{2}}p_0$ since $N^{m-1}p_0>1$.

Next, we calculate the variance of $V$ under $H_1$. As a first step, it's easy to check that
$\mathbb{V}[(\hat{p}_0-p_0)^2] \leq \frac{p_0^2}{N^{(m)2}}+\frac{p_0}{N^{(m)3}} $.
First, we have the following decomposition
\begin{eqnarray}\nonumber
&&\sum_{i_1=1}^N\left[W_{i_1*}-(m-1)!(N-1)^{(m-1)}p_0\right]^2
=\sum_{\substack{i_1\neq i_2\neq\cdots\neq i_m\\ i_1\neq j_2\neq\cdots\neq j_m}}(A_{i_1i_2\ldots i_m}-p_0)(A_{i_1j_2\ldots j_m}-p_0)\\   \nonumber
 &=&\sum_{i_1\in S,S_{i_1},S_{i_1}}(A_{i_1i_2\ldots i_m}-p_0)(A_{i_1j_2\ldots j_m}-p_0)+\sum_{i_1\in S,\bar{S}_{i_1},S_{i_1}}(A_{i_1i_2\ldots i_m}-p_0)(A_{i_1j_2\ldots j_m}-p_0)\\   \nonumber
 &&+\sum_{i_1\in S,S_{i_1},\bar{S}_{i_1}}(A_{i_1i_2\ldots i_m}-p_0)(A_{i_1j_2\ldots j_m}-p_0) +\sum_{i_1\in S,\bar{S}_{i_1},\bar{S}_{i_1}}(A_{i_1i_2\ldots i_m}-p_0)(A_{i_1j_2\ldots j_m}-p_0)\\   \label{p0eq1}
 && +\sum_{i_1\notin S,U_{i_1},U_{i_1}}(A_{i_1i_2\ldots i_m}-p_0)(A_{i_1j_2\ldots j_m}-p_0).
\end{eqnarray}
We find the variance of each term in the right hand side of (\ref{p0eq1}). For the first term, one has
\begin{eqnarray}\nonumber
&&\sum_{i_1\in S,S_{i_1},S_{i_1}}(A_{i_1i_2\ldots i_m}-p_0)(A_{i_1j_2\ldots j_m}-p_0)\\ \nonumber
&=&\sum_{i_1\in S,S_{i_1},S_{i_1}}(A_{i_1i_2\ldots i_m}-p_1)(A_{i_1j_2\ldots j_m}-p_1)+\sum_{i_1\in S,S_{i_1},S_{i_1}}(A_{i_1i_2\ldots i_m}-p_1)(p_1-p_0)  \\ \nonumber
&&+\sum_{i_1\in S,S_{i_1},S_{i_1}}(p_1-p_0)(A_{i_1j_2\ldots j_m}-p_1)+\sum_{i_1\in S,S_{i_1},S_{i_1}}(p_1-p_0)^2.
\end{eqnarray}
Straightforward calculation yields
\begin{eqnarray*}\nonumber
\mathbb{V}[\sum_{i_1\in S,S_{i_1},S_{i_1}}(A_{i_1i_2\ldots i_m}-p_1)(p_1-p_0) ]&=&(p_1-p_0)^2p_1(1-p_1)m!n^{(m)}(m-1)!^2(n-1)^{(m-1)2},   \\ \nonumber
\mathbb{E}[\sum_{i_1\in S,S_{i_1},S_{i_1}}(A_{i_1i_2\ldots i_m}-p_1)(A_{i_1j_2\ldots j_m}-p_1)]&=&p_1(1-p_1)m!n^{(m)},   \\ \nonumber
\mathbb{E}[\sum_{i_1\in S,S_{i_1},S_{i_1}}(A_{i_1i_2\ldots i_m}-p_1)(A_{i_1j_2\ldots j_m}-p_1)]^2
&=&p_1^2(1-p_1)^2m!(m-1)!(n-1)^{(m-1)}n^{(m)}\\
&&+p_1^2(1-p_1)^2m!^2n^{(m)2}+m!n^{(m)}p_1(1-p_1)O(1).
\end{eqnarray*}
Combining with the fact that $p_1=p_0^{\prime}+(p_0-p_0^{\prime})\frac{N^{(m)}}{n^{(m)}}$ and $p_0^{\prime}\leq p_0\leq p_1$, we can get
\begin{eqnarray*}
&&\mathbb{V}\left[\sum_{i_1\in S,S_{i_1},S_{i_1}}(A_{i_1i_2\ldots i_m}-p_0)(A_{i_1j_2\ldots j_m}-p_0)\right]\\
&\leq& N^mp_0+N^{2m-1}p_0^2+n^{3m-2}(p_1-p_0^{\prime})^3+N^{m-1}n^{2m-1}p_0(p_1-p_0^{\prime})^2.
\end{eqnarray*}

For the third term in (\ref{p0eq1}), we decompose it as
\begin{eqnarray}\nonumber
&&\sum_{i_1\in S,S_{i_1},\bar{S}_{i_1}}(A_{i_1i_2\ldots i_m}-p_0)(A_{i_1j_2\ldots j_m}-p_0)\\ \nonumber
&=&\sum_{i_1\in S,S_{i_1},\bar{S}_{i_1}}(A_{i_1i_2\ldots i_m}-p_1)(A_{i_1j_2\ldots j_m}-p_0^{\prime})+\sum_{i_1\in S,S_{i_1},\bar{S}_{i_1}}(A_{i_1i_2\ldots i_m}-p_1)(p_0^{\prime}-p_0)  \\ \label{p0eq2}
&&+\sum_{i_1\in S,S_{i_1},\bar{S}_{i_1}}(p_1-p_0)(A_{i_1j_2\ldots j_m}-p_0^{\prime})+\sum_{i_1\in S,S_{i_1},\bar{S}_{i_1}}(p_1-p_0)(p_1-p_0^{\prime}).
\end{eqnarray}
By finding the variance of each term in (\ref{p0eq2}), we can get
\begin{eqnarray*}\nonumber
&&\mathbb{V}\left[\sum_{i_1\in S,S_{i_1},\bar{S}_{i_1}}(A_{i_1i_2\ldots i_m}-p_0)(A_{i_1j_2\ldots j_m}-p_0)\right]\\
&\leq& N^{2m+1}p_0^2+n^mN^{(m-1)2}p_1(p_0^{\prime}-p_0) ^2+p_0^{\prime}(p_1-p_0^{\prime})^2N^{m-1}n^{2m-1}\\  \nonumber
&\leq&N^{2m+1}p_0^2+ n^{3m-2}(p_1-p_0^{\prime})^3+N^{m-1}n^{2m-1}p_0(p_1-p_0^{\prime})^2.
\end{eqnarray*}
The variance of the second term in (\ref{p0eq1}) can be treated similar as the variance of the third term.
As to the variances of the fourth and fifth terms in (\ref{p0eq1}), one has
\begin{eqnarray}\nonumber
\mathbb{V}\left[\sum_{i_1\in S,\bar{S}_{i_1},\bar{S}_{i_1}}(A_{i_1i_2\ldots i_m}-p_0)(A_{i_1j_2\ldots j_m}-p_0)\right]&\leq&N^{2m-1}p_0^2+2nN^{3m-3}p_0^{\prime}(p_0-p_0^{\prime})^2\\  \nonumber
&\leq& N^{2m-1}p_0^2+2p_0(p_1-p_0^{\prime})^2N^{m-1}n^{2m-1},\\  \nonumber
\mathbb{V}\left[\sum_{i_1\notin S,U_{i_1},U_{i_1}}(A_{i_1i_2\ldots i_m}-p_0)(A_{i_1j_2\ldots j_m}-p_0)  \right]&\leq& (N^{2m-1}p_0^2+N^mp_0)+2p_0^{\prime}(p_0^{\prime}-p_0) ^2N^{3m-2}\\  \nonumber
&\leq& N^{2m-1}p_0^2+N^{m-1}n^{2m-1}p_0(p_1-p_0^{\prime})^2.
\end{eqnarray}
Based on the above, we get that
\begin{eqnarray*}
&&\mathbb{V}\left[\sum_{i_1=1}^N\left[W_{i_1*}-(m-1)!(N-1)^{(m-1)}p_0\right]^2\right]\\
&\leq& N^mp_0+N^{2m-1}p_0^2+n^{3m-2}(p_1-p_0^{\prime})^3+N^{m-1}n^{2m-1}p_0(p_1-p_0^{\prime})^2.
\end{eqnarray*}
In summary, we have
\[\mathbb{V}[V]\leq N^{2m-3}p_0^2+N^{m-3}n^{2m-1}p_0(p_1-p_0^{\prime})^2+\frac{n^{3m-2}}{N^2}(p_1-p_0^{\prime})^3.\]
Recall that $\mathbb{E}V\gg N^{\frac{2m-3}{2}}p_0$. Then $p_0N^{m-1}\ll n^{2m-1}(p_1-p_0^{\prime})^2$ and $\sqrt{N}\ll n^{2m-1}(p_1-p_0^{\prime})^2$. It follows that
\[\mathbb{E}V\gg \sqrt{N^{m-3}n^{2m-1}p_0(p_1-p_0^{\prime})^2},\,\,\,\, \mathbb{E}V\gg \sqrt{\frac{n^{3m-2}}{N^2}(p_1-p_0^{\prime})^3}.\]
\end{proof}
Before proving Theorem \ref{thm:6} \ref{thm:6:III},
we need the following lemma.
\begin{Lemma}\label{cyclelemma1}
Under the conditions of Theorem \ref{thm:6},
if (\ref{thm8eq}) holds, then 
\[
\frac{p_1-p_0^{\prime}}{\sqrt{p_0}}\left(\frac{n^2}{N}\right)^{\frac{m+1}{4}}\rightarrow\infty.
\]
\end{Lemma}
\begin{proof}
[Proof of Lemma \ref{cyclelemma1}] If $p_0^{\prime}\asymp p_0$, the proof is complete. Suppose $p_0^{\prime}\ll (p_1-p_0^{\prime})\frac{n^m}{N^m}$. Then $p_1-p_0^{\prime}\gg \frac{N^m}{n^m}p_0^{\prime}$.
Under condition (\ref{ucond:1}), it's easy to verify that $n^{m-1}p_0^{\prime}> \frac{1}{\sqrt{N}}$. As a result, $p_0^{\prime}>\frac{1}{\sqrt{N}n^{m-1}}\gg \frac{1}{\sqrt{N}N^{\frac{m-1}{2}}}$ if $n=o(\sqrt{N})$.
In this case, since $N^mp_0^{\prime}\gg N^{\frac{m}{2}}$,
\[\frac{(p_1-p_0^{\prime})^2}{p_0}\left(\frac{n^2}{N}\right)^{\frac{m+1}{2}}\asymp \frac{(p_1-p_0^{\prime})^2}{(p_1-p_0^{\prime})\frac{n^m}{N^m}}\left(\frac{n^2}{N}\right)^{\frac{m+1}{2}}=(p_1-p_0^{\prime})nN^{\frac{m-1}{2}}\gg N^mp_0^{\prime}\frac{N^{\frac{m-1}{2}}}{n^{m-1}}\rightarrow\infty.\]
\end{proof}

\begin{proof}[Proof of Theorem \ref{thm:6} \ref{thm:6:III}] Under $H_0$, we prove $\mathcal{T}_2=O_P(1)$. 
Note that
$\mathbb{E}(\hat{p}_0-p_0)^2=\frac{p_0(1-p_0)}{N^{(m)}}$.
Hence, $\hat{p}_0=p_0(1+o_p(1))$ and $\mathcal{T}_2=T_n(1+o_p(1))$, where
\[T_n=\frac{\sum_{\text{$i_1,\ldots,i_{m+1}$ are pairwise distinct}}(A_{i_1\ldots i_m}-\hat{p}_0)(A_{i_2\ldots i_{m+1}}-\hat{p}_0)}{\sqrt{(m+1)!N^{(m+1)}p_0^2(1-p_0)^2}}.\]
Straightforward calculation yields
\begin{eqnarray*}\label{thm8eq2}
T_n&=&\frac{\sum_{\text{$i_1,\ldots,i_{m+1}$ are pairwise distinct}}(A_{i_1\ldots i_m}-p_0)(A_{i_2\ldots i_{m+1}}-p_0)}{\sqrt{(m+1)!N^{(m+1)}p_0^2(1-p_0)^2}}-\frac{\sqrt{(m+1)!N^{(m+1)}}}{\sqrt{p_0^2(1-p_0)^2}}(\hat{p}_0-p_0)^2\\  \nonumber
&=&\frac{\sum_{\text{$i_1,\ldots,i_{m+1}$ are pairwise distinct}}(A_{i_1\ldots i_m}-p_0)(A_{i_2\ldots i_{m+1}}-p_0)}{\sqrt{(m+1)!N^{(m+1)}p_0^2(1-p_0)^2}}-O_p\left(\frac{1}{N^{\frac{m-1}{2}}}\right)=T_{1n}+o_p(1).
\end{eqnarray*}
Note that the variance of $T_{1n}$ under $H_0$ is equal to one, then $T_{1n}=O_p(1)$ and hence $T_n=O_p(1)$.

Next, we show $T_n$ is unbounded in probability under $H_1$. 
By (\ref{unk1}), it follows that
\begin{eqnarray*}
&&\frac{\sqrt{(m+1)!N^{(m+1)}}}{\sqrt{p_0^2(1-p_0)^2}}(\hat{p}_0-p_0)^2\\
&=&O_P\left(\frac{1}{N^{\frac{m-1}{2}}}+\frac{N^{\frac{m+1}{2}}n^m(p_1-p_0^{\prime})^2}{N^{2m}p_0}\right)=O_P\left(\frac{1}{N^{\frac{m-1}{2}}}+\frac{(p_1-p_0^{\prime})}{N^{\frac{m-1}{2}}}(1-\frac{p_0^{\prime}}{p_0})\right)=o_p(1).
\end{eqnarray*}
Then $T_n=T_{1n}+o_p(1)$ under $H_1$. 

Next, we prove $T_{1n}$ is unbounded in probability.
For any given $S\subset [N]$ with $|S|=n$, we have
\begin{eqnarray*}
T_{1n}&=&\frac{\sum_{S,S}(A_{i_1\ldots i_m}-p_0)(A_{i_2\ldots i_{m+1}}-p_0)}{\sqrt{(m+1)!N^{(m+1)}p_0^2(1-p_0)^2}}+\frac{\sum_{\bar{S},S}(A_{i_1\ldots i_m}-p_0)(A_{i_2\ldots i_{m+1}}-p_0)}{\sqrt{(m+1)!N^{(m+1)}p_0^2(1-p_0)^2}}\\
&&+\frac{\sum_{S,\bar{S}}(A_{i_1\ldots i_m}-p_0)(A_{i_2\ldots i_{m+1}}-p_0)}{\sqrt{(m+1)!N^{(m+1)}p_0^2(1-p_0)^2}}+\frac{\sum_{\bar{S},\bar{S}}(A_{i_1\ldots i_m}-p_0)(A_{i_2\ldots i_{m+1}}-p_0)}{\sqrt{(m+1)!N^{(m+1)}p_0^2(1-p_0)^2}}\\
&=&T_{an}+T_{bn}+T_{cn}+T_{dn}.
\end{eqnarray*}

For $T_{dn}$, we have
\begin{eqnarray*}
T_{dn}&=&\frac{\sum_{\bar{S},\bar{S}}(A_{i_1\ldots i_m}-p_0^{\prime})(A_{i_2\ldots i_{m+1}}-p_0^{\prime})}{\sqrt{(m+1)!N^{(m+1)}p_0^2(1-p_0)^2}}+\frac{2(p_0^{\prime}-p_0)\sum_{\bar{S},\bar{S}}(A_{i_1\ldots i_m}-p_0^{\prime})}{\sqrt{(m+1)!N^{(m+1)}p_0^2(1-p_0)^2}}\\
&&+\frac{\sum_{\bar{S},\bar{S}}(p_0^{\prime}-p_0)^2}{\sqrt{(m+1)!N^{(m+1)}p_0^2(1-p_0)^2}}=T_{d1n}+T_{d2n}+T_{d3n}.
\end{eqnarray*}
The variance of $T_{dn}$ is bounded by
\[V_1(T_{dn})\leq V_1(T_{d1n})+V_1(T_{d2n})\leq \frac{(N^{m+1}-n^{m+1})p_0^{\prime 2}(1-p_0^{\prime})^2}{N^{m+1}p_0^2}+\frac{N^{m+2}(p_0^{\prime}-p_0)^2p_0^{\prime }(1-p_0^{\prime})}{N^{m+1}p_0^2}.\]

Since $T_{bn}$ and $T_{cn}$ are of the same order, we only consider $T_{bn}$.
\begin{eqnarray*}
T_{bn}&=&\frac{\sum_{\bar{S},S}(A_{i_1\ldots i_m}-p_0^{\prime})(A_{i_2\ldots i_{m+1}}-p_1)}{\sqrt{(m+1)!N^{(m+1)}p_0^2(1-p_0)^2}}+\frac{(p_1-p_0)\sum_{\bar{S},S}(A_{i_1\ldots i_m}-p_0^{\prime})}{\sqrt{(m+1)!N^{(m+1)}p_0^2(1-p_0)^2}}\\
&&+\frac{(p_0^{\prime}-p_0)\sum_{\bar{S},S}(A_{i_2\ldots i_{m+1}}-p_1)}{\sqrt{(m+1)!N^{(m+1)}p_0^2(1-p_0)^2}}+\frac{\sum_{\bar{S},S}(p_0^{\prime}-p_0)(p_1-p_0)}{\sqrt{(m+1)!N^{(m+1)}p_0^2(1-p_0)^2}}=T_{b1n}+T_{b2n}+T_{b3n}+T_{b4n}.
\end{eqnarray*}

The variance of $T_{bn}$ is bounded by
\begin{eqnarray*}
V_1(T_{bn})&\leq & V_1(T_{b1n})+V_1(T_{b2n})+V_1(T_{b3n})\leq 
\frac{(N-n)n^mp_0^{\prime }(1-p_0^{\prime})p_1(1-p_1)}{N^{m+1}p_0^2}\\
&+&\frac{Nn^{m+1}p_0^{\prime }(1-p_0^{\prime })(p_1-p_0)^2}{N^{m+1}p_0^2}+\frac{N^2n^{m}p_1(1-p_1)(p_0^{\prime }-p_0)^2}{N^{m+1}p_0^2}.
\end{eqnarray*}

For $T_{an}$, we have
\begin{eqnarray*}
T_{an}&=&\frac{\sum_{S,S}(A_{i_1\ldots i_m}-p_1)(A_{i_2\ldots i_{m+1}}-p_1)}{\sqrt{(m+1)!N^{(m+1)}p_0^2(1-p_0)^2}}+\frac{2(p_1-p_0)\sum_{S,S}(A_{i_1\ldots i_m}-p_1)}{\sqrt{(m+1)!N^{(m+1)}p_0^2(1-p_0)^2}}\\
&&+\frac{\sum_{S,S}(p_1-p_0)^2}{\sqrt{(m+1)!N^{(m+1)}p_0^2(1-p_0)^2}}=T_{a1n}+T_{a2n}+T_{a3n}.
\end{eqnarray*}

The variance of $T_{an}$ is bounded by
\begin{eqnarray*}
V_1(T_{an})&\leq & V_1(T_{a1n})+V_1(T_{a2n})\leq 
\frac{n^{m+1}p_1^2(1-p_1)^2}{N^{m+1}p_0^2}+\frac{n^{m+2}p_1(1-p_1)(p_1-p_0)^2}{N^{m+1}p_0^2}.
\end{eqnarray*}

By Lemma \ref{cyclelemma1}, the expectation of $T_{1n}$ under $H_1$ is
\begin{eqnarray*}
\mathbb{E}_1(T_{1n})&=&\frac{\sum_{\bar{S},\bar{S}}(p_0^{\prime}-p_0)^2+2\sum_{\bar{S},S}(p_0^{\prime}-p_0)(p_1-p_0)+\sum_{S,S}(p_1-p_0)^2}{\sqrt{(m+1)!N^{(m+1)}p_0^2(1-p_0)^2}}\\
&\asymp&\frac{n^{m+1}(p_1-p_0)^2}{\sqrt{(m+1)!N^{(m+1)}p_0^2(1-p_0)^2}}\\
&\asymp&\frac{(p_1-p_0^{\prime})^2}{p_0}\left(\frac{n^2}{N}\right)^{\frac{m+1}{2}}\rightarrow\infty.
\end{eqnarray*}

Let $V_1=V_1(T_{1n})$ under $H_1$. Then we have
\[T_{1n}=\sqrt{V_1}\left(\frac{T_{1n}-\mathbb{E}_1T_{1n}}{\sqrt{V_1}}+\frac{\mathbb{E}_1T_{1n}}{\sqrt{V_1}}\right).\]

If $V_1=O(1)$, then $T_{1n}=O_p(1)+\mathbb{E}_1(T_{1n})\rightarrow\infty$ in probability. If $p_1\asymp p_0^{\prime}$, it's easy to check that $V_1=O(1)$. In the following, we assume $V_1\rightarrow\infty$ and $p_1\gg p_0^{\prime}$.

Note that $V_1(T_{d2n})\leq V_1(T_{d1n})\leq 1$ and $V_1(T_{b3n})\ll V_1(T_{a2n})$.
Besides, if $p_0\asymp p_0^{\prime}$, then, 
\[V_1(T_{b1n})=
\frac{(N-n)n^mp_0^{\prime }(1-p_0^{\prime})p_1(1-p_1)}{N^{m+1}p_0^2}\leq \frac{n^m}{N^m}\frac{p_1}{p_0}=O(1).\]
If $p_0^{\prime}\ll (p_1-p_0^{\prime})\frac{n^m}{N^m}\asymp p_0$, then
\[V_1(T_{b1n})=
\frac{(N-n)n^mp_0^{\prime }(1-p_0^{\prime})p_1(1-p_1)}{N^{m+1}p_0^2}\asymp \frac{N^m}{n^m}\frac{p_0^{\prime}}{p_1}=O(1).\]
As a result, $V_1(T_{b1n})=O(1)$ and $V_1\rightarrow\infty$ implies 
\[
V_1=\max\{V_1(T_{b2n}), V_1(T_{a1n}), V_1(T_{a2n})\}\rightarrow\infty.
\]

If $V_1(T_{b2n})\gg V_1(T_{a1n}), V_1(T_{b2n})\gg V_1(T_{a2n})$, then $V_1=V_1(T_{b2n})\rightarrow\infty$ and
\[\frac{V_1(T_{b2n})}{V_1(T_{a1n})}\asymp\frac{Np_0^{\prime}(p_1-p_0)^2}{p_1^2}\asymp Np_0^{\prime}\rightarrow\infty.\]
In this case, 
\[T_{1n}=\sqrt{V_1}\left(\frac{T_{1n}-\mathbb{E}_1T_{1n}}{\sqrt{V_1}}+\frac{\mathbb{E}_1T_{1n}}{\sqrt{V_1}}\right)=\sqrt{V_1}\left(\frac{T_{b2n}}{\sqrt{V_1}}+o_p(1)\right)+\mathbb{E}_1T_{1n}.\]
By the central limit theorem, we conclude that $\frac{T_{b2n}}{\sqrt{V_1}}$ converges in distribution to $N(0,1)$ if $n^{m-1}Np_0^{\prime}\rightarrow\infty$. Consequently, $\mathbb{P}(|T_{1n}|>c)=1-o(1)$ for any fixed constant $c>0$.

If $V_1(T_{a2n})\gg V_1(T_{a1n}), V_1(T_{a2n})\gg V_1(T_{b2n})$, then $V_1=V_1(T_{a2n})\rightarrow\infty$ and
\[\frac{V_1(T_{a2n})}{V_1(T_{a1n})}\asymp\frac{n^{m+2}p_1(p_1-p_0)^2}{n^{m+1}p_1^2}\asymp np_1\rightarrow\infty.\]
In this case, 
\[T_{1n}=\sqrt{V_1}\left(\frac{T_{1n}-\mathbb{E}_1T_{1n}}{\sqrt{V_1}}+\frac{\mathbb{E}_1T_{1n}}{\sqrt{V_1}}\right)=\sqrt{V_1}\left(\frac{T_{a2n}}{\sqrt{V_1}}\right)(1+o_p(1))+\mathbb{E}_1T_{1n}.\]
By the central limit theorem, we conclude that $\frac{T_{a2n}}{\sqrt{V_1}}$ converges in distribution to $N(0,1)$ if $n^{m}p_1\rightarrow\infty$. Consequently, $\mathbb{P}(|T_{1n}|>c)=1-o(1)$ for any fixed constant $c>0$.

If $V_1(T_{a1n})\gg V_1(T_{a2n}), V_1(T_{a1n})\gg V_1(T_{b2n})$, then $V_1=V_1(T_{a1n})\rightarrow\infty$.
Under condition (\ref{ucond:1}), it's easy to verify that $n^{m-1}p_0^{\prime}> \frac{1}{\sqrt{N}}$. As a result, $p_0^{\prime}>\frac{1}{\sqrt{N}n^{m-1}}\gg \frac{1}{\sqrt{N}N^{\frac{m-1}{2}}}$ if $n=o(\sqrt{N})$. Consequently, $N^{m+1}p_0^2>N^{m+1}p_0^{\prime 2}\gg N\rightarrow\infty$. Hence, $V_1=V_1(T_{a1n})\rightarrow\infty$ implies $n^{m+1}p_1^2\rightarrow\infty$. In this case, 
\[T_{1n}=\sqrt{V_1}\left(\frac{T_{1n}-\mathbb{E}_1T_{1n}}{\sqrt{V_1}}+\frac{\mathbb{E}_1T_{1n}}{\sqrt{V_1}}\right)=\sqrt{V_1}\left(\frac{T_{a1n}}{\sqrt{V_1}}\right)(1+o_p(1))+\mathbb{E}_1T_{1n}.\]
Since $T_{a1n}$ can be expressed as a sum of martingale differences,
by the Martingale central limit theorem \cite{HH14}, we get that $\frac{T_{a1n}}{\sqrt{V_1}}$ converges in distribution to $N(0,1)$ if $n^{m+1}p_1^2\rightarrow\infty$ (See \cite{YN20} for a proof when $m=3$). Consequently, $\mathbb{P}(|T_{1n}|>c)=1-o(1)$ for any fixed constant $c>0$. 

If $ V_1(T_{a2n})\asymp V_1(T_{b2n})\gg V_1(T_{a1n})$, then 
\[\frac{V_1(T_{a2n})}{V_1(T_{a1n})}=np_1\asymp Np_0^{\prime}=\frac{V_1(T_{b2n})}{V_1(T_{a1n})}\rightarrow\infty.\]
In this case, $V_1=V_1(T_{b2n})$ is bounded since
\[V_1(T_{b2n})=\frac{Nn^{m+1}p_0^{\prime }(1-p_0^{\prime })(p_1-p_0)^2}{N^{m+1}p_0^2}\leq \frac{p_1^2p_0^{\prime }Nn^{m+1}}{N^{m+1}p_0^{\prime 2}}\asymp\frac{Nn^mp_1}{N^m}\ll\frac{p_1}{N^{\frac{m}{2}-1}} =o(1).\]

If $ V_1(T_{a1n})\asymp V_1(T_{b2n})\gg V_1(T_{a2n})$, then 
\[\frac{V_1(T_{a2n})}{V_1(T_{a1n})}=np_1=o(1),\  1\asymp Np_0^{\prime}=\frac{V_1(T_{b2n})}{V_1(T_{a1n})}.\]
In this case, $V_1=V_1(T_{b2n})$ is bounded since
\[V_1(T_{b2n})=\frac{Nn^{m+1}p_0^{\prime }(1-p_0^{\prime })(p_1-p_0)^2}{N^{m+1}p_0^2}\leq \frac{p_1^2p_0^{\prime }Nn^{m+1}}{N^{m+1}p_0^{\prime 2}}\ll\frac{Nn^mp_1}{N^m}\ll\frac{p_1}{N^{\frac{m}{2}-1}} =o(1).\]

If $ V_1(T_{a1n})\asymp V_1(T_{a2n})\gg V_1(T_{b2n})$, then 
\[\frac{V_1(T_{a2n})}{V_1(T_{a1n})}=np_1\asymp 1,\  o(1)=Np_0^{\prime}=\frac{V_1(T_{b2n})}{V_1(T_{a1n})}.\]
In this case, $n^{m+1}p_1^2\rightarrow\infty$ and 
\[T_{1n}=\sqrt{V_1}\left(\frac{T_{1n}-\mathbb{E}_1T_{1n}}{\sqrt{V_1}}+\frac{\mathbb{E}_1T_{1n}}{\sqrt{V_1}}\right)=\sqrt{V_1}\left(\frac{T_{a1n}+T_{a2n}}{\sqrt{V_1}}+o_p(1)\right)+\mathbb{E}_1T_{1n}.\]
Since $T_{a1n}+T_{a2n}$ can be expressed as a sum of martingale differences,
by the Martingale central limit theorem, we get that $\frac{T_{a1n}+T_{a2n}}{\sqrt{V_1}}$ converges in distribution to $N(0,1)$ if $n^{m+1}p_1^2\rightarrow\infty$. Consequently, $\mathbb{P}(|T_{1n}|>c)=1-o(1)$ for any fixed constant $c>0$. 

If $ V_1(T_{a1n})\asymp V_1(T_{a2n})\asymp V_1(T_{b2n})$, then 
\[\frac{V_1(T_{a2n})}{V_1(T_{a1n})}=np_1\asymp Np_0^{\prime}=\frac{V_1(T_{b2n})}{V_1(T_{a1n})}\asymp1.\]
In this case, $p_1\asymp\frac{1}{n}$ and $p_0^{\prime}\asymp\frac{1}{N}$. Consequently,
\[ V_1(T_{a1n})=
\frac{n^{m+1}p_1^2(1-p_1)^2}{N^{m+1}p_0^2}\leq \frac{n^{m+1}p_1^2}{N^{m+1}p_0^{\prime 2}}\asymp\frac{n^{m-2}}{N^{m-2}}=o(1).\]
Hence, $V_1=O(1)$. Proof is completed.
\end{proof}

\end{document}